\documentclass{article} 
\usepackage{graphicx, amssymb, color}
\usepackage{amsmath}

\usepackage{graphicx}
\usepackage{enumerate}
\usepackage{algorithm}
\usepackage{algorithmic}
\usepackage{bbm}

\newtheorem{thm}{Theorem}[section]
\newtheorem{cor}[thm]{Corollary}
\newtheorem{lem}[thm]{Lemma}
\newtheorem{prop}[thm]{Proposition}

\newtheorem{rmk}[thm]{Remark}

\newcommand{\set}[1]{\left\{#1\right\}}

\newcommand{\LT}{\mathcal{L}}

\newenvironment{proof}{\paragraph{Proof:}}{\hfill$\square$}
\usepackage{mathtools}

\usepackage{authblk}

\title{Explicit Asymptotics on First Passage Times of Diffusion Processes}
\author{Angelos Dassios\thanks{Department of Statistics, London School of Economics, Houghton Street, London WC2A 2AE, UK, a.dassios@lse.ac.uk} }
\author{Luting Li\thanks{Department of Statistics, London School of Economics, Houghton Street, London WC2A 2AE, UK, l.li27@lse.ac.uk}}
\affil{Department of Statistics, London School of Economics}

\begin{document}
\date{\today}
\maketitle

\begin{abstract}
We introduce a unified framework for solving first passage times of time-homogeneous diffusion processes. According to the killed version potential theory and the perturbation theory, we are able to deduce closed-form solutions for probability densities of single-sided level crossing problem. The framework is applicable to diffusion processes with continuous drift functions, and a recursive system in the frequency domain has been provided. Besides, we derive a probabilistic representation for error estimation. The representation can be used to evaluate deviations in perturbed density functions. In the present paper, we apply the framework to Ornstein-Uhlenbeck and Bessel processes to find closed-form approximations for their first passage times; another successful application is given by the exponential-Shiryaev process \cite{dassios2018economic}. Numerical results are provided at the end of this paper.

	\smallskip

\smallskip


\end{abstract}

\section{Introduction}
The interest of understanding the \emph{first passage time} (FPT) could be traced back to the early 20th century \cite{bachelier1900theorie, schrodinger1915theorie}. Known also as the \emph{first hitting time}, the FPT defines a random time that a stochastic process would visit a predefined state. The phenomenon of uncertainty in time is often observed from natural or social science. Therefore, within a century the FPT has been actively studied in economics, physics, biology, etc. \cite{redner2001guide, ricciardi1999outline, novikov2003time, dassios2018recursive}.

Depending on various types of underlying processes and hitting boundaries, the FPT itself consists of a large cluster of different research. We refer to \cite{siegert1951first, arbib1965hitting, blake1973level, novikov1981martingale} for a non-conclusive review. Among those research, especially in the area of mathematical finance and insurance, single-sided constant-barrier crossing problem is one of the most commonly studied, e.g. \cite{baldi1999pricing, dassios2016joint}. A general approach for solving such problem starts with finding the \emph{Laplace transform} (LT) of the FPT density (FPTD). The LT usually comes from a unique solution to a second order non-homogeneous ODE with Dirichlet-type boundary values \cite{doob2012classical, ikeda2014stochastic}. For many familiar diffusion processes, the LTs have been solved and are listed in \cite{borodin2012handbook}. However, those LTs usually are expressed in terms of special functions and only a few of them have explicit inverse transforms. Therefore, many efforts have been made on the numerical inverse side. We refer to \cite{abate2006unified} for more detail. Alternatively, using spectral theorem on linear operators \cite{ito2012diffusion, kent1982spectral, kent1980eigenvalue} one can simplify the original LT. Under certain circumstances, closed-form FPTDs could be acquired through series representations \cite{linetsky2004computing, alili2005representations}. But people may find the spectral decomposition approach has convergence issue for small $t$. In the present paper, our object is to apply perturbation theory and solve explicit FPTDs for general single-side level crossing problem.

Consider a filtered probability space generated by Brownian motion $\left(\Omega,\mathcal{F},\left\{\mathcal{F}_t\right\}_{t\ge 0},\mathbb{P}\right)$. Let $\mathcal{D}$ be an open interval on $\mathbb{R}$ and $h(\cdot)$ be a continuous function defined on $\mathcal{D}$. Our underlying process is from a class of SDEs. We require these SDEs have at least weak solutions and are strong Markov:
\begin{equation}\label{eqn:basic}
X_t=\epsilon h(X_t)dt + dW_t,\ X_0=x \in \mathcal{D}.
\end{equation}
Under our settings, $\epsilon$ is a real parameter and it should properly define $\set{X_t}_{t\ge 0}$ on the domain. For the convenience of deduction, we set the volatility to be constant. If a time-homogeneous diffusion coefficient $\sigma(x)$ is given, one may refer to \cite[Theorem 1.6]{revuz2013continuous} to retrieve an SDE in \eqref{eqn:basic} by using time-changed Brownian motion. Also, consider a hitting level $a\in \mathbb{R}$, we specify two types of boundaries on $\mathcal{D}$:
\[
\partial\mathcal{D}^u_a:=\left\{a,+\infty\right\},\ \partial\mathcal{D}^l_a:=\left\{-\infty, a\right\},
\] 
namely boundaries for upper- and lower-regions. For shorthand, we use $\partial\mathcal{D}_a$ to represent single sided boundaries without labelling direction. By suppressing $x$ and $a$, we define the FPT of $\set{X_t}_{t\ge 0}$ from $x$ to $a$ through
\[
\tau:=\inf\left\{t> 0:X_t\in \partial\mathcal{D}_a\right\}.
\]
Note that the Brownian filtration $\set{\mathcal{F}_t}_{t\ge 0}$ is continuous on both sides. Therefore according to \cite{peskir2006optimal}, $\tau$ is well defined (\emph{regular}). In addition, for $x\in \mathcal{D}$ it is guaranteed that\footnote{The notation $\mathbb{P}_x(\cdot):=\mathbb{P}(\cdot|\mathcal{F}_0)=\mathbb{P}(\cdot|X_0=x)$ follows from the Markov property.} $\mathbb{P}_x\left(\tau>0\right)=1$.

For those FPTs which are \emph{almost surely} (\emph{a.s.}) finite, \emph{i.e.} $\mathbb{P}_x\left(\tau<+\infty\right)=1$, we are interested in acquiring their explicit distributions. Clearly, when $h(x)\equiv 0$ (standard Brownian motion) the distribution of $\tau$ is given by inverse Gaussian (or inverse Gamma, equivalently) \cite{borodin2012handbook}. However, for most of non-trivial drifts, there is no closed-form solution. An example is $h(x)=x$ and which corresponds to the Ornstein-Uhlenbeck (OU) process. In this case, the explicit density is only available by restricting $a=0$ \cite{going2003clarification}.

In this paper, we apply perturbation technique \cite{holmes2012introduction} to solve Dirichlet-type \emph{boundary value problems} (BVPs). By inverting the perturbed LTs from the frequency domain, where those LTs usually have much simpler forms, we then are able to derive closed-form densities in the time domain. The main contribution of this paper is to provide a unified recursive framework for solving the single barrier hitting problem. And according to the killed version of potential theory \cite{peskir2006optimal}, we prove convergence and error estimation results. As illustrations, we show perturbed FPTDs on OU and Bessel processes in this paper. An application on the Bubble (exponential-Shiryaev) process has been discussed in \cite{dassios2018economic}. At the end of this paper, theoretical part is confirmed via numerical exercises.

The rest of the paper is organised as follows: section 2 introduces main results; sections 3 and 4 demonstrate applications on OU and Bessel processes; in section 5 we show benchmarking exercises based on OU process (more results can be found in the appendix); section 6 concludes.

\section{Main Results}\label{sec13}
\subsection{Perturbed Dirichlet Problem}\label{epsilon_ode}
We follow our previous settings. Further let $C^2$ be the collection of functions with second order continuous derivatives. For any $f\in C^2$, also assume the infinitesimal generator $\mathcal{A}f(x)$ of $\set{X_t}_{t\ge 0}$ exists for all $x\in \mathcal{D}$. Then $\mathcal{A}\cdot$ is explicitly given by
\[
\mathcal{A}f(x)=\epsilon h(x)f^{'}(x) +\mathcal{G}f(x),
\]
where $\mathcal{G}\cdot$ is the infinitesimal generator for standard Brownian motion
\[
\mathcal{G}f(x)=\frac{1}{2}f^{''}(x);
\]
the prime notation refers to the derivative \emph{w.r.t.} $x$. And unless specified, we will use this conventional notation throughout the paper.

Consider $\beta\in\mathbb{C}$ with $Real(\beta)\ge 0$, define,
\begin{equation}\label{LTderivation}
f(x,\beta):=\mathbb{E}_x\left[e^{-\beta \tau}V\left(X_\tau\right)\right],
\end{equation}
where $V\left(\cdot\right)$ is a finite function. The first step of our work is to find a proper BVP which is satisfied by $f(x,\beta)$. To see this, we first need to show $\set{X_t}_{t\ge 0}$ is continuous over all stopping times. In fact, consider a sequence of stopping times $\left\{\sigma_n\right\}_{1\le n\le +\infty}$ such that
\[
\lim_{n\uparrow +\infty}\sigma_n=:\sigma.
\]
Since $\set{X_t}_{t\ge 0}$ has continuous path so 
\[
\lim_{n\uparrow +\infty}X_{\sigma_n}=X_\sigma.
\]
On the other hand, by our assumption, $\set{X_t}_{t\ge 0}$ is a strong Markov process. According to the killed version of potential theory \cite{peskir2006optimal}, $f(x,\beta)$ is the unique solution to the following Dirichlet problem:
\begin{equation}
\label{eqn0}
\mathcal{A}f(x)=\beta f(x),\ x\in \mathcal{D}.
\end{equation}
Moreover, the corresponding boundary conditions are given by
\begin{equation}
\label{bv0}
f(\partial\mathcal{D}_a)=\underline{V}.
\end{equation}
In our notation $\underline{V}:=\left[V(a),V(\pm \infty)\right]^T$ is a vector of the boundary values depending on the direction of crossing. Refer to \eqref{LTderivation}, by setting $V(a)=1$ and $V(\pm\infty)=0$ we immediately find that the solution to BVP \eqref{eqn0} and \eqref{bv0} is the LT for the density function of $\tau$:
\[
f(x,\beta)=\mathbb{E}_x\left[e^{-\beta \tau}\right].
\]

In the second step, we apply perturbations on $\epsilon$ and find perturbed BVPs accordingly. The perturbation approach is a common technique in solving asymptotics for complex systems. It has been successfully applied in quantum physics and mathematical finance \cite{schrodinger1926quantisierung, dassios2010perturbed, fouque2011multiscale}. Traditionally, it is required that the perturbation parameter should be small. However, we will show this is not necessary in our case.

For abbreviation, we ignore the function arguments in following contents. By default all operations are \emph{w.r.t.} $x$. Consider a sequence of $C^2$-functions $\left\{f_i\right\}_{i\ge 0}$ such that $f$ can be expressed as
\begin{equation}
\label{ptfb}
f=\sum_{i=0}^\infty \epsilon^i f_i.
\end{equation}
Substitute \eqref{ptfb} into \eqref{eqn0}, we have
\begin{equation}\label{eqn:firstsubstitution}
\sum_{i=0}^{\infty}\epsilon^i\left(\epsilon h f_i^{'} +\mathcal{G}f_i\right)=\sum_{i=0}^\infty \epsilon^i \beta f_i,\ \forall x\in \mathcal{D}.
\end{equation}
Rearrange terms in \eqref{eqn:firstsubstitution}, we further get
\[
\mathcal{G}f_0-\beta f_0+\sum_{i=1}^{\infty}\epsilon^i\left(h f_{i-1}^{'}+\mathcal{G}f_i-\beta f_i\right)=0,\ \forall x\in \mathcal{D}.
\]
Note that, by extracting the $0$-th order and assigning proper boundary conditions we can have the BVP for standard Brownian motion (where the LT inverse is already known). Higher orders can be solved via a recursive system which accumulates information from $f_0$ and the drift function $h$. 

Denote the BVP with $i=0$ by $o(1)$ term, by assigning same boundary conditions as in the initial problem, we have
\begin{align*}
o(1):\ & \mathcal{G}f_0=\beta f_0,\ x\in \mathcal{D}\\
& f_0(\partial \mathcal{D}_a)=[1,0]^T.
\end{align*}
For $i\ge 1$, we use the notation $o(\epsilon^i)$ and define
\begin{align*}
o\left(\epsilon^i\right):\ & \mathcal{G}f_i=\beta f_i-h\cdot f_{i-1}^{'},\ x\in \mathcal{D}\\
& f_i(\partial \mathcal{D}_a)=[0,0]^T.
\end{align*}
Based on the fact that the solution to the initial BVP is unique, one can check by solving the recursive system to infinite orders (not necessarily for small $\epsilon$) the sequence $\set{f_i}_{i\ge 0}$ reproduces the initial solution $f$, \emph{i.e.} equation \eqref{ptfb} always holds true. However, in practice, it is not realistic of having infinite order solutions. Also under very common circumstances, the convergence of the series may not be guaranteed. Therefore we need to decide a truncation order and estimate the corresponding error.

\subsection{Truncation Error and Convergence}\label{subsec23}
Further introduce some notations. Let $N\ge 1$ be a fixed integer, and for $i=1,...,N$ we denote the $N$-th order truncation of initial LT by
\begin{equation}\label{eqn:truncatedLT}
f^N:=\sum_{i=0}^N \epsilon^i f_i.
\end{equation}
Assume inverse LTs for $f$, $f^N$ and $\partial_xf_N(x,\beta)$ exist, and denote by
\[
p_\tau(t)=\LT^{-1}\left\{f(\beta)\right\}(t),
\]
\[
p_\tau^N(t)=\sum_{i=0}^N\epsilon^i\LT^{-1}\left\{f_i(\beta)\right\}(t),
\]
and
\begin{equation}\label{eqn:eqndefeta}
\eta(x,t)=\LT^{-1}\left\{\partial_xf_N(x,\beta)\right\}(t),
\end{equation}
respectively. Define the difference (absolute error) between two FPTDs by
\begin{equation}\label{errorfunctiondef}
q_\tau(t):=p_\tau(t)-p_\tau^N(t),
\end{equation}
then we have the following result.

\begin{prop}[Probabilistic Representation for the Truncation Error]\label{prop21}
For all $t\in (0,+\infty)$ and all $\beta\in\mathbb{C}$ with $Real(\beta)>0$, if
\begin{equation}\label{eqn:eqnboundnesscondition}
\int_0^{+\infty}e^{-\beta t}\mathbb{E}_x\left[\int_0^{\tau\wedge t}\left|h(X_u)\eta\left(X_u,t-u\right)\right|du\right]dt <+\infty,
\end{equation}
then
\begin{equation}\label{eqn:errorestimateqfunc}
q_\tau(t)=\epsilon^{N+1}\mathbb{E}_x\left[\int_0^{\tau\wedge t}h(X_u)\eta\left(X_u,t-u\right)du\right].
\end{equation}
Further, if for some constant $M<+\infty$
\begin{equation}\label{eqn:assumpboundnedness}
\left|\mathbb{E}_x\left[\int_0^{\tau\wedge t}h(X_u)\eta\left(X_u,t-u\right)du\right]\right|\le M,
\end{equation}
then
\begin{equation}\label{eqn:errorestimateboundedness}
\left|q_\tau(t)\right|\le \epsilon^{N+1}M.
\end{equation}
\end{prop}
\begin{proof}
Let $\eta\left(x,t\right)$ be defined as in \eqref{eqn:eqndefeta}, and introduce
\[
\tilde{q}_\tau(t):=\epsilon^{N+1}\mathbb{E}_x\left[\int_0^{\tau\wedge t}h(X_u)\eta\left(X_u,t-u\right)du\right].
\]
We first show $\tilde{q}_\tau(t)$ is the inverse LT of $f-f^N$. Then by the uniqueness of (inverse) LT, $\tilde{q}_\tau(t)$ is the error function in \eqref{errorfunctiondef}. Consider the LT of $\tilde{q}_\tau(t)$. By \eqref{eqn:eqnboundnesscondition} and the Fubini's theorem, we have
\begin{align*}
\int_0^\infty e^{-\beta t} \mathbb{E}_x\left[\int_0^{\tau\wedge t}h(X_u)\eta\left(X_u,t-u\right)du\right]dt&=\mathbb{E}_x\left[\int_0^{\tau} h(X_u)\int_0^\infty e^{-\beta t}\bold{1}_{\left\{u\le  t\right\}}\eta\left(X_u,t-u\right)dt du\right]\\
&=\mathbb{E}_x\left[\int_0^\tau h(X_u) \LT\left\{\bold{1}_{\left\{u\le  t\right\}}\eta\left(X_u,t-u\right)\right\}(\beta)du\right].
\end{align*}
According to the fact
\[
\LT\left\{\bold{1}_{\left\{u\le t\right\}}\eta(x,t-u)\right\}(\beta)=e^{-\beta u}\LT\left\{\eta(x,t)\right\}(\beta),
\]
and by the definition of $\eta(x,t)$, we therefore conclude
\begin{equation}
\label{LTform}
\LT\left\{\tilde{q}_\tau(t)\right\}(\beta)=\epsilon^{N+1}\mathbb{E}_x\left[\int_0^\tau h(X_u)e^{-\beta u}\frac{\partial }{\partial x}f_N(X_u,\beta)du\right].
\end{equation}

In the second part below, we show the right-hand side of \eqref{LTform} indeed is $f-f^N$. Let 
\[
Q( x,\beta):=f(x,\beta)-f^N(x,\beta).
\]
By the linearity of LT we have
\[
Q(x,\beta)=\LT\left\{p_\tau(t)\right\}(\beta)-\LT\left\{p^N_\tau(t)\right\}(\beta)=\LT\left\{q_\tau(t)\right\}(\beta).
\]
Since $f$ and $f^N$ are both in $C^2$, so is $Q$. Applying the operator $\mathcal{A}$ on $Q$ yields
\begin{align*}
\mathcal{A}Q-\beta Q &=\mathcal{A}f-\beta f-\left(\mathcal{A}f^N-\beta f^N\right)\\
&=0-\left[\mathcal{G}f_0-\beta f_0+\sum_{i=1}^N\epsilon^i\left(\mathcal{G}f_0-\beta f_0+h\cdot f_{i-1}^{'}\right)+\epsilon^N\cdot \epsilon h f_N^{'}\right]\\
&=-\epsilon^{N+1}h\cdot f_N^{'}.
\end{align*}
Note that $f$ and $f^N$ share the same boundary conditions, so for $Q(x)$ we have
\[
Q\left(\partial \mathcal{D}_a\right)=[0,0]^{T}.
\]
According to \cite{peskir2006optimal}, the ODE of $Q(x)$ is the killed version of Dirichlet problem and its solution has the following probabilistic representation:
\[
Q(x,\beta)=\epsilon^{N+1}\mathbb{E}_x\left[\int_0^\tau e^{-\beta u}h(X_u)\frac{\partial}{\partial x}f_N( X_u,\beta)du\right].
\]
This is indeed the right-hand side of \eqref{LTform}. By the uniqueness of BVP solution and the uniqueness of (inverse) LT, we conclude that
\[
q_\tau(t)=\tilde{q}_\tau(t).
\]
In the end, \eqref{eqn:errorestimateboundedness} is a direct result from assumption \eqref{eqn:assumpboundnedness} and \eqref{eqn:errorestimateqfunc}.
\end{proof}
\\

\begin{rmk}\label{rmk21}
For small $\epsilon$ and under condition \eqref{eqn:assumpboundnedness}, by proposition \ref{prop21} we see the $N$-th order perturbed FPTD converges to the true density at $O(\epsilon^{N+1})$. Moreover, this convergence is uniform on $t$. On the other hand, even for large $\epsilon$, one can always use \eqref{eqn:errorestimateqfunc} to check error levels.
\end{rmk}

\subsection{Recursions under Frequency Domain}\label{subsec21}
In this section, we provide a general mechanism for solving recursive BVPs. For simplicity\footnote{For an arbitrary hitting level $a$ we can use the affine transformation to retrieve the $0$-hitting case. Although not always, for the situation of hitting from below ($x<0$) we can consider the mirror reflection of $\set{X_t}_{t\ge 0}$.} we consider the FPT hitting $0$ from above, \emph{i.e.} 
\begin{equation}
\label{myhitting}
\tau:=\inf\left\{t\ge 0: X_t=0\bigg|X_0=x>0\right\}.
\end{equation}
Under this treatment the domain is specified by $\mathcal{D}=(0,+\infty)$. We suppress the notation $a$ (note that $a=0$), and denote the boundaries by $\partial \mathcal{D}:=\partial \mathcal{D}_a=\set{0,+\infty}$. 
\begin{lem}[Laplace Transform for the FPTD of Brownian Motion]\label{lemma31}
The unique solution to the $o(1)$ BVP is given by
\[
f_0(x,\beta)=e^{-\sqrt{2\beta}x}.
\]
\end{lem}
\begin{proof}
This is the standard result from \cite{borodin2012handbook}.
\end{proof}
\\

\begin{lem}[Recursive Solution to $o\left(\epsilon^i\right)$]\label{lemma32}
For $i\ge 1$, let
\[
\gamma:=\sqrt{2\beta}.
\] 
The unique solution to $o\left(\epsilon^i\right)$ is given by
\[
f_i(x,\beta)=f_0(x,\beta)\cdot\left[\int_0^x2 e^{2\gamma y}\int_0^y h(z) k_i(z,\beta) e^{-2\gamma z}dz dy + \frac{e^{2\gamma x}-1}{\gamma}c_i(\beta)\right],
\]
where 
\[
k_i(x,\beta)=2\gamma\int_{0}^{x}e^{2\gamma y}\int_{0}^{y}h
 \left( z \right) { k_{i-1}} (z,\beta) {e}^{-2\gamma z}
dzdy
-2e^{2\gamma x}\int_{0}^{x}h
 \left( y \right) { k_{i-1}} (y,\beta)e^{-2 \gamma y}dy
 -\left({e^{2\gamma x}+1}\right)c_{i-1}(\beta)
\]
and
\[
c_i(\beta)=-\gamma \lim_{x\uparrow +\infty}\left(e^{-2\gamma x}\int_0^x2 e^{2\gamma y}\int_0^y h(z) k_i(z,\beta) e^{-2\gamma z}dz dy \right).
\]
\end{lem}
\begin{proof}
The uniqueness of $f_i$ follows from the Dirichlet-type BVP \cite{peskir2006optimal}. Consider $f_i$ is of the following form
\begin{equation}\label{eqnfigiexpression}
f_i:=f_0g_i.
\end{equation}
Then substituting \eqref{eqnfigiexpression} into $o\left(\epsilon^i\right)$-ODE yields 
\begin{equation}
\label{eqnOrg}
\frac{1}{2}g^{''}_i-\gamma g_i^{'}=h\left[\gamma g_{i-1}-g_{i-1}^{'}\right].
\end{equation}
Note $g_{i-1}$ and its derivative are determined in the $i$-th order. We denote by
\begin{equation}
\label{eqnK}
k_i:=\gamma g_{i-1}-g_{i-1}^{'}.
\end{equation}
Then equation \eqref{eqnOrg} can be rewritten as
\begin{equation}
\label{eqnsimplified}
\frac{1}{2}g^{''}_i-\gamma g_i^{'}=hk_i.
\end{equation}
Multiply $e^{-2\gamma x}$ and take integrations on both sides of \eqref{eqnsimplified}, we have
\[
\int \frac{1}{2}g^{''}_ie^{-2\gamma x}dx -\int \gamma g_i^{'}e^{-2\gamma x}dx = \int hk_ie^{-2\gamma x}dx + C_1.
\] 
Apply integral by parts, for the left-hand side we get
\begin{align*}
\int \frac{1}{2}g^{''}_ie^{-2\gamma x}dx -\int \gamma g_i^{'}e^{-2\gamma x}dx&= \frac{1}{2}g_i^{'}e^{-2\gamma x}+\gamma \int g_i^{'} e^{-2\gamma x}dx -\int \gamma g_i^{'}e^{-2\gamma x}dx= \frac{1}{2}g_i^{'}e^{-2\gamma x}.
\end{align*}
Further multiply $2e^{2\gamma x}$ and take integrals on both sides,
\begin{equation}\label{eqngeneralsolutiongi}
g_i=\int_{A_1}^x 2e^{2\gamma y} \left[\int_{A_2}^y hk_ie^{-2\gamma z}dz + C_1\right]dy +C_2.
\end{equation}
\emph{W.l.o.g.}, we let $A_1=A_2=0$, and by considering the boundary condition at $x=0$ we have $C_2=0$. On the other hand, note that for fixed $i$, $C_1$ is a function related to order $i$ and depends on $\beta$, so we rewrite $c_i(\beta):=C_1$. Further simplify $g_i$ with new notations, we get
\begin{equation}
\label{eqnGi}
g_i=\int_0^x2 e^{2\gamma y}\int_0^y h k_i e^{-2\gamma z}dz dy + \frac{e^{2\gamma x}-1}{\gamma}c_i.
\end{equation}
To determine $c_i$, we use the boundary condition at $+\infty$. By solving $\lim_{x\uparrow+\infty}f_i(x,\beta)=0$ we get
\begin{equation}
\label{eqnCi}
c_i=-\gamma \lim_{x\uparrow +\infty}\left(e^{-2\gamma x}\int_0^x2 e^{2\gamma y}\int_0^y h k_i e^{-2\gamma z}dz dy \right).
\end{equation}
Substitute \eqref{eqnGi} into equation \eqref{eqnK}, and after standard calculations we get
\begin{equation}
\label{eqnKrecur}
k_{i}=
2\gamma\int_{0}^{x}e^{2\gamma y}\int_{0}^{y}h { k_{i-1}}  {e}^{-2\gamma z}
dzdy
-2e^{2\gamma x}\int_{0}^{x}h
  { k_{i-1}} e^{-2 \gamma y}dy
  - \left({e^{2\gamma x}+1}\right)c_{i-1}.
\end{equation}
This concludes our proof.
\end{proof}
\\

\begin{rmk}
Potentially, using lemma \ref{lemma32} we can solve the LT of perturbed FPTD to orders as many as we wish. With the help of symbolic calculation softwares (e.g. Maple, Python), calculations would become even simpler.
\end{rmk}

\section{Ornstein-Uhlenbeck Process}\label{sec4}
The OU process was first introduced to describe the velocity of a particle that follows a Brownian motion movement \cite{uhlenbeck1930theory}. Later the process appears widely in neural science \cite{wan1982neuronal, lansky1995comparison} and mathematical finance \cite{vasicek1977equilibrium, jarrow1995pricing, schwartz1997stochastic, heston1993closed}, etc. According to our setting, the $h$-function for OU process is given by
\begin{equation}\label{hfunctionforOU}
h(x)=\theta -x,
\end{equation}
where $\theta$ is the equilibrium parameter in the OU model. Note that the function above is a first order polynomial and which is still a polynomial under integration and differential. Since our recursive framework mainly involves those two operations, therefore we may expect a beautiful result from perturbation. 

Refer to \cite{uhlenbeck1930theory, wang1945theory}, the OU SDE has a unique and strong solution. Moreover, as a strong Markov process it is recurrent and continuous in probability. Our framework therefore can be applied. Let $\epsilon>0$ be the mean-reverting rate, the infinitesimal generator is given by
\[
\mathcal{A}f=\frac{1}{2}\frac{\partial ^2f}{\partial x^2}+\epsilon(\theta - x)\frac{\partial f}{\partial x}.
\]
Consider hitting from above in \eqref{myhitting}. The solution of initial BVP \eqref{eqn0} and \eqref{bv0} is given by \cite{borodin2012handbook}. And for the numerical inverse of initial LT one can find in \cite{abate2006unified}. Under a special case that $\theta=0$, the explicit FPTD is given by \cite{going2003clarification, alili2005representations}. Now we solve the explicit FPTD using perturbation technique. Follow lemma \ref{lemma32}, for $i=1$ we immediately have
\[
f_1(x,\beta)=e^{-\gamma x}\left(\frac{x^2}{2}+\frac{(1-2\theta\gamma)x}{2\gamma}\right).
\]

\begin{prop}[Recursive Polynomial]
\label{prop:prop1}
For $i\ge 2$, the solution of $o(\epsilon^i)$-BVP for the OU process is given by
\[
f_i(x,\beta)=f_0(x,\beta) g_i(\gamma x),
\]
where
\begin{equation}\label{eqn:OUgi}
g_i(y):=\sum_{j=1}^{2i}p^{(i)}_jy^j.
\end{equation}
The coefficients $\{p^{(i)}_j: i\ge 2, 1\le j\le 2i\}$ are given by:
\[
\begin{cases}
j=2i:& p^{(i)}_{2i}=\frac{p^{(i-1)}_{2i-2}}{2 i\gamma^2}\\
\\
j=2i-1:& p^{(i)}_{2i-1}=\frac{1}{ (2i-1)\gamma^2}p^{(i-1)}_{2i-3}+\left(\frac{1}{2\gamma^2}-\frac{\gamma \theta +(2i-2)}{ (2i-1)\gamma^2}\right)p^{(i-1)}_{2i-2}\\
\\
2< j< 2i-1:& p^{(i)}_{j}=\frac{1}{2}(j+1)p^{(i)}_{j+1}+\frac{1}{\gamma^2 j}p^{(i-1)}_{j-2}-\frac{j-1+\gamma\theta}{\gamma^2j}p^{(i-1)}_{j-1}+\frac{\theta}{\gamma}p^{(i-1)}_{j}\\
\\
j=2:& p^{(i)}_{2}=\frac{3}{2}p^{(i)}_{3}-\frac{1+\gamma\theta}{2\gamma^2}p^{(i-1)}_{1}+\frac{\theta}{\gamma}p^{(i-1)}_{2}\\
\\
j=1:& p^{(i)}_{1}=p^{(i)}_{2}+\frac{\theta}{\gamma}p^{(i-1)}_{1}
\end{cases}.
\]
\end{prop}

\begin{proof}
One can try to solve coefficients from the recursive algorithm in lemma \ref{lemma32}. Another approach, which is much simpler, is to directly derive results from ODE \eqref{eqnsimplified}. We leave details of proof to readers.
\end{proof}
\\

Proposition \ref{prop:prop1} enables us to expand perturbed LT to arbitrary orders. However, our final aim is to find its inverse while coefficients in the proposition are functions of the parameter $\beta$. In order to avoid unnecessary symbolic calculations, we further decompose the coefficients $\left\{p^{(i)}_j:i\ge 0,\ 1\le j\le 2i \right\}$.

\begin{lem}[Recursive Polynomial Decomposition]
\label{lemma2}
For each $p_j^{(i)}$ in $\left\{p_j^{(i)}:i\ge 1,\ 1\le j \le 2i\right\}$, there exists a triply-indexed real sequence $$\left\{c_k^{(i,j)}: i\ge 1,\ 1\le j \le 2i,\ 0\le k \le (2i-j)\wedge i\right\}$$ such that
\[
p_j^{(i)}=\sum_{k=0}^{(2i-j)\wedge i}c_k^{(i,j)}\left(\frac{1}{\gamma}\right)^{2i-k}(\theta)^k.
\]
\end{lem}

\begin{proof}
The recursive structure of $\set{c_k^{(i,j)}}$ can be seen in corollaries \ref{cor2} and \ref{cor3} in appendix. One can immediately get corollary \ref{cor2} by comparing it with $f_1$ and $f_2$. For corollary \ref{cor3}, results are acquired by substituting the decomposition in lemma \ref{lemma2} into proposition \ref{prop:prop1}. The existence in lemma \ref{lemma2} is then proved by corollaries \ref{cor2} and \ref{cor3}.
\end{proof}
\\

\begin{rmk}
Note that $\left\{c_k^{(i,j)}\right\}$ is parameter-independent. Therefore we can pre-calculate the sequence and save it in memory. Later, this will help in enhancing the FPTD computational efficiency.
\end{rmk}

\begin{prop}[$N$-th Order Perturbed FPTD of OU Process]
\label{prop402}
Let $N\in \mathbb{Z}^+$, the $N$-th order perturbed downward FPTD of OU process is
\[
{p}_{\tau}^{(N)}(t)=\frac{e^{-\frac{x^2}{4t}}}{\sqrt{2\pi}}\sum_{n=0}^{2N-1}h_nt^{\frac{n}{2}-1}D_{-n+1}\left(\frac{x}{\sqrt{t}}\right),
\]
where
\[
h_n=\sum_{\substack{i,j,k;\\2i-j-k=n}}\epsilon^i c_k^{(i,j)}\theta^kx^j,
\]
and $D_\cdot(\cdot)$ is the \emph{parabolic cylinder function}.
\end{prop}

\begin{proof}
Express the truncated LT using \eqref{eqn:truncatedLT} and \eqref{eqnfigiexpression}. According to proposition \ref{prop:prop1} and lemma \ref{lemma2}, 
\begin{equation}\label{eqnfN}
f^N=\sum_{i=0}^{N}\epsilon^i e^{-\gamma x}\sum_{j=1}^{2i}{\left(\sum_{k=0}^{(2i-j)\wedge i}{c_k^{(i,j)}\left(\frac{1}{\gamma}\right)^{2i-k-j}\theta^k}\right) x^j}.
\end{equation}
For $0\le n \le 2N-1$, define 
\[
\hat{h}_n:=2^{-\frac{n}{2}}h_n.
\]
Then after standard calculations, we can write \eqref{eqnfN} as
\begin{equation}\label{fNLT}
f^N=\sum_{n=0}^{2N-1}\hat{h}_n\cdot e^{-\sqrt{2\beta} x}\beta^{-\frac{n}{2}}.
\end{equation}
Note that $\hat{h}_n$ is independent to $\beta$. By referring to \cite{bateman1954tables} we find for $0\le n\le 2N-1$,
\begin{equation}\label{DinverseTransform}
\LT^{-1}\set{e^{-\sqrt{2\beta} x}\beta^{-\frac{n}{2}}}(t)=\frac{2^{\frac{n}{2}}t^{\frac{n}{2}-1}}{\sqrt{2\pi}}e^{-\frac{x^2}{4t}}D_{-n+1}\left(\frac{x}{\sqrt{t}}\right).
\end{equation}
This immediately gives us the result.
\end{proof}
\\

\begin{rmk}
According to \cite{borodin2012handbook}, the initial LT of the OU FPTD is given by the ratio of two parabolic cylinder functions. In the end, by applying perturbations we find the FPTD itself (which is an inverse LT) is expressed as a series of parabolic cylinder functions with the first argument to be integers. 
\end{rmk}

The integer-valued $D$-function is closely related to Hermite polynomials \cite[12.7 (i)]{lozier2003nist}. Based on the fact \cite[12.7.1]{lozier2003nist} that 
\begin{equation}\label{eqn:D0}
D_0(z)=e^{-\frac{1}{4}z^2},
\end{equation}
and according to \cite[12.7.2]{lozier2003nist} we have 
\begin{equation}\label{eqn:D1}
D_1\left(\frac{x}{\sqrt{t}}\right)=\frac{x}{\sqrt{t}}e^{-\frac{x^2}{4t}}.
\end{equation}
Using expressions above we can write
\begin{equation}\label{new27}
{p}_{\tau}^{(N)}(t)=(h_0+\frac{h_1}{x}t)p_\tau^{(0)}(t)+\frac{e^{-\frac{x^2}{4t}}}{\sqrt{2\pi}}\sum_{n=2}^{2N-1}h_nt^{\frac{n}{2}-1}D_{-n+1}\left(\frac{x}{\sqrt{t}}\right),
\end{equation}
where $p_\tau^{(0)}(t)$ is the FPTD of standard Brownian motion of hitting $0$, and which follows the inverse Gamma distribution. From this point of view, the $N$-th order OU perturbed FPTD is indeed the FPTD of Brownian motion plus higher order corrections.

\begin{prop}[Tail Asymptotics of OU Perturbed FPTD]\label{prop:tailasmptotocsou}
For $N\ge 1$, tail asymptotics of the $N$-th order perturbed OU FPTD are given by
\begin{align*}
p_\tau^{(N)}(t)&\sim p_\tau^{(0)}(t),\ \text{as }t\downarrow 0^{+},\tag*{(Left Tail)}\\
p_\tau^{(N)}(t)&\sim t^{N-\frac{3}{2}}\tag*{(Right Tail)},\ \text{as }t\uparrow +\infty.
\end{align*}
\end{prop}

\begin{proof}
We start from the left tail. When $t\downarrow 0^{+}$, it is clear from \eqref{new27} that
\[
(h_0+\frac{h_1}{x}t)p_\tau^{(0)}(t)\sim p_\tau^{(0)}(t).
\]
We need to further check the asymptotic for the series of $D$-functions. Denote by $z:=\frac{x}{\sqrt{t}}$ and $a=n-\frac{3}{2}$ ($2\le n\le 2N-1$). Refer to \cite[19.8.1, page 689]{abramowitz1964handbook} and \cite[12.9.1]{lozier2003nist}, when $z\uparrow +\infty$ (\emph{i.e.} $t\downarrow 0$) we have
\begin{align*}
D_{-a-\frac{1}{2}}(z)\sim &e^{-\frac{1}{4}z^2}z^{-a-\frac{1}{2}}\left(1+O(\frac{1}{z^2})\right).
\end{align*}
Therefore 
\begin{align*}
\frac{1}{\sqrt{2\pi}}\sum_{n=2}^{2N-1}h_nt^{\frac{n}{2}-1}e^{-\frac{x^2}{4t}}D_{-n+1}\left(\frac{x}{\sqrt{t}}\right)\sim&\sum_{n=2}^{2N-1} t^{n-\frac{3}{2}}e^{-\frac{x^2}{2t}} \left(1+O(t)\right)\\
\sim& p_\tau^{(0)}(t) (t^2+o(t^2)).
\end{align*}
This proves the left tail result.

Now we consider the right tail. Again by \eqref{new27} we immediately have
\begin{equation}\label{eqn:firstorderdiverge}
\left(h_0+\frac{h_1}{x}t\right)p^{(0)}_\tau(t)\sim \frac{h_1}{x}tp^{(0)}_\tau(t)\sim t^{-\frac{1}{2}},\ \text{for }t\uparrow +\infty.
\end{equation}
For the series of $D$-functions we use \cite[19.12.3]{abramowitz1964handbook}, and 
\[
D_{-a-\frac{1}{2}}(z)=\sqrt{\pi}2^{-\frac{1}{2}a}e^{-\frac{1}{4}z^2}\left(\frac{2^{-\frac{1}{4}}}{\Gamma\left(\frac{3}{4}+\frac{1}{2}a\right)}M\left(\frac{1}{2}a+\frac{1}{4},\frac{1}{2},\frac{1}{2}z^2\right)-\frac{2^{\frac{1}{4}}z}{\Gamma\left(\frac{1}{4}+\frac{1}{2}a\right)}M\left(\frac{1}{2}a+\frac{3}{4},\frac{3}{2},\frac{1}{2}z^2\right)\right),
\]
where $M(\cdot,\cdot,\cdot)$ is the Kummer's function (confluent hypergeometric function $\ _1 F_1(\cdot,\cdot,\cdot)$). Introduce the notations 
\[
\zeta(a,s):= \sqrt{\pi}\frac{2^{-\frac{1}{4}-\frac{1}{2}a}}{\Gamma\left(\frac{3}{4}+\frac{1}{2}a\right)}\frac{(\frac{1}{2}a+\frac{1}{4})_s}{(\frac{1}{2})_s2^ss!},
\] 
and
\[
\xi(a,s):= -\sqrt{\pi}\frac{2^{\frac{1}{4}-\frac{1}{2}a}}{\Gamma\left(\frac{1}{4}+\frac{1}{2}a\right)}\frac{(\frac{1}{2}a+\frac{3}{4})_s}{(\frac{3}{2})_s2^ss!},
\] 
where $(\cdot)_s:=\Pi_{k=0}^{s-1}(\cdot + k)$. According to \cite[13.2.2]{lozier2003nist}, we can reformulate the $D$-function as
\[
D_{-n+1}\left(\frac{x}{\sqrt{t}}\right)=e^{-\frac{x}{4t}}\left(\sum_{s=0}^\infty\zeta(n-\frac{3}{2},s)\frac{x^{2s}}{t^s}+\sum_{s=0}^\infty\xi(n-\frac{3}{2},s)\frac{x^{2s+1}}{t^{s+\frac{1}{2}}}\right).
\]
Now let $t\uparrow +\infty$, it then yields
\begin{align*}
t^{\frac{n}{2}-1}e^{-\frac{x^2}{4t}}D_{-n+1}\sim &t^{\frac{n}{2}-1}\left(\sum_{s=0}^\infty\zeta(n-\frac{3}{2},s)\frac{x^{2s}}{t^s}+\sum_{s=0}^\infty\xi(n-\frac{3}{2},s)\frac{x^{2s+1}}{t^{s+\frac{1}{2}}}\right).
\end{align*}
Note that for fixed $n$ the leading term in above is $t^{\frac{n}{2}-1}$. Therefore, by considering the highest order $n=2N-1$ in the $D$-function series, we get
\begin{equation}\label{eqn:secondorderdiverge}
\frac{1}{\sqrt{2\pi}}\sum_{n=2}^{2N-1}h_nt^{\frac{n}{2}-1}e^{-\frac{x^2}{4t}}D_{-n+1}\left(\frac{x}{\sqrt{t}}\right)\sim t^{N-\frac{3}{2}}.
\end{equation}
This concludes our proof.
\end{proof}
\\

From the right tail asymptotic we find for $N\ge 2$, the perturbed FPTD would diverge at $t=+\infty$. In the case $N=1$, though $p_\tau^{(1)}(t)$ converges to $0$, due to the fat-tail effect of $t^{-\frac{1}{2}}$ we still expect an infinite integral. Therefore the total probability of perturbed FPTD is infinity for all $N\ge 1$. Indeed, the limit value of LT tells for all $N\ge 1$,
\[
\left|\int_0^\infty p_\tau^{(N)}(t)dt\right|=\lim_{\beta\downarrow 0^{+}}\left|f^N(\beta)\right|=\infty.
\]

On the other hand, the left tail asymptotic shows equivalence between OU and Brownian motion FPTDs. Opposite to the spectral decomposition \cite{linetsky2004computing, alili2005representations}, our analysis indicates the perturbation may not work well for large $t$, but it provides smooth densities for small $t$.

\begin{lem}[$\eta$-Function for OU Process]\label{lemma410}
For $N\ge 1$, the $\eta$-function is given by
\begin{equation}\label{enq:etaforOU}
\eta(x,t)=\frac{e^{-\frac{x^2}{4t}}}{\sqrt{2\pi}}\sum_{n=0}^{2N-1}t^{\frac{n}{2}-1}\left((\partial_x l_n)D_{-n+1}\left(\frac{x}{\sqrt{t}}\right)-\frac{l_n}{\sqrt{ t}}D_{-n+2}\left(\frac{x}{\sqrt{t}}\right)\right),
\end{equation}
where
\begin{equation}\label{defLn}
l_n=\sum_{\substack{j,k;\\j+k=2N-n}} c_k^{(N,j)}\theta^k\cdot x^j.
\end{equation}
\end{lem}
\begin{proof}
We start with the definition of $\eta$-function (proposition \ref{prop21}). The partial derivative of $f_N(x,\beta)$ is given by
\[
\partial_xf_N(x,\beta)=e^{-\sqrt{2\beta} x}\sum_{n=0}^{2N-1}2^{-\frac{n}{2}} \partial_x(l_n)\beta^{-\frac{n}{2}}-2^{-\frac{n-1}{2}} l_n\beta^{-\frac{n-1}{2}}.
\]
The rest of proof is concluded by using \eqref{DinverseTransform} again.
\end{proof}\\

\begin{prop}[First Order Error Estimation and Convergence of OU Process]\label{firorderOU}
For $N=1$ and all $t\in [0,+\infty)$, the error estimation \eqref{eqn:errorestimateqfunc} of perturbed OU FPTD is valid. Moreover, for any $T>0$ and $t\in[0,T]$, the perturbation is $o(\epsilon^2)$-accurate.
\end{prop}
\begin{proof}
Recall $h$-function in \eqref{hfunctionforOU} and $\eta$-function in \eqref{enq:etaforOU}. In the first part of the proof, we show \eqref{eqn:eqnboundnesscondition} in proposition \ref{prop21} is satisfied. As a sufficient condition to \eqref{eqn:eqnboundnesscondition}, we need to find a bound
\[
\mathbb{E}_x\left[\int_0^{\tau\wedge t}|h(X_u)\eta(X_u,t-u)|du\right]\le M(t)
\]
for all $t\in [0,\infty)$, and such that $M(t)$ grows less than exponential.

For all fixed $x\in\mathcal{D}$, one can check that $h(x)\eta(x,t)\rightarrow 0$ as $t\uparrow +\infty$. Let $K>1$ be a large and fixed constant. $h(x)\eta(x,t)$ is continuous, so there exists a constant $M_1$, \emph{w.l.o.g.}\footnote{Note that, $h(0)\eta(0,t)$ and $h(x)\eta(x,0)$ are both well-defined for $t$ and $x$ on open intervals. The only singularity is generated by $(x,t)=(0,0)$. However, under the probability space, $X_u=0$ and $t-u=0$ only happens with $t=\tau$, while $P_x(\tau=t)=0$. Therefore, we can define $h(0)\eta(0,0)=0$.}, for all $(x,t)\in [0,K]\times [0,+\infty)$, such that
\begin{equation}\label{smallerbound}
\left|h(x)\eta(x,t)\bold{1}_{\set{x\le K}}\right|\le M_1\bold{1}_{\set{x\le K}}.
\end{equation}
Note that $\mathbb{P}_x(X_u\le K)\le 1$, therefore \eqref{smallerbound} yields
\begin{align}
\mathbb{E}_x\left[\int_0^{\tau\wedge t}|h(X_u)\eta(X_u,t-u)|\bold{1}_{\set{X_u\le K}}du\right]&\le \left[\int_0^{t} M_1\mathbb{E}_x\left[\bold{1}_{\set{X_u\le K}}\right]du\right]\label{firstpartbound111_1}\\
&\le M_1\cdot t\label{firstpartbound111_2}.
\end{align}
On the other hand, for $x>K$, by lemma \ref{firorderOUlemma} in the appendix, we have
\begin{equation}\label{secondpartbound111}
\mathbb{E}_x\left[\int_0^{\tau\wedge t}|h(X_u)\eta(X_u,t-u)|\bold{1}_{\set{X_u>K}}du\right]\le C_0+C_1\sqrt{t}+C_2 t.
\end{equation}
Combine \eqref{firstpartbound111_2} and \eqref{secondpartbound111}, we get
\[
\mathbb{E}_x\left[\int_0^{\tau\wedge t}|h(X_u)\eta(X_u,t-u)|du\right]\le C_0+C_1\sqrt{t}+(C_2+M_1) t.
\]
This concludes the first part proof. The second part of the proof is given by showing, for $t\in [0,T]$,
\[
\left|\mathbb{E}_x\left[\int_0^{\tau\wedge t}h(X_u)\eta(X_u,t-u)du\right]\right|\le C_0+C_1\sqrt{T}+(C_2+M_1) T.
\]
\end{proof}
\\

\section{Bessel Process}\label{sec5} 
Bessel process was introduced in \cite{mckean1960bessel} as the norm of an $n$-dimensional Brownian motion. Denoted by $BES(n)$ (sometimes by $BES(\nu)$ with $\nu = \frac{n-2}{2}$), Bessel process has been discussed extensively in \cite[Chapter XI]{revuz2013continuous}. In mathematical finance, the family of Bessel processes is closely\footnote{Squared Bessel process appears in the Constant Elasticity of Variance (CEV) model and the Cox-Ingersoll-Ross(CIR) process (or Heston model).} related to models in short rates and stochastic volatilities \cite{carr2006jump, cox2005theory, chen1992pricing, heston1993closed, dassios2018azema}. 

In this section, we consider the class of Bessel processes with orders $n=1+\epsilon$, where $-1<\epsilon<1$. For $BES(1+\epsilon)$, the $h$-function is specified by 
\begin{equation}\label{hufnctionforBessel}
h(x)=\frac{1}{2x}.
\end{equation}
The domain of $BES(1+\epsilon)$ is $\mathcal{D}=[0,+\infty)$. According to \cite{ito2012diffusion, kent1978some}, $\set{\infty}$ is a natural boundary for all $\epsilon\in\mathcal{R}$; for $-1<\epsilon<1$ ($0<n<2$), $\set{0}$ is a regular boundary with both types of exit and entrance. We assume the process makes instantaneous reflection at $0$. Refer to \cite{kent1978some}, the infinitesimal generator then is well defined for all $C^2$-functions on $\mathcal{D}$. Denote it by
\[ 
\mathcal{A} f(x)=\frac{\epsilon}{2x}\frac{\partial f}{\partial x}+\frac{1}{2}\frac{\partial ^2 f}{\partial x^2}.
\]
We consider a general hitting level $0<a<X_0$ rather than $a=0$. By suppressing $a$ we redefine
\[
\tau=\inf\left\{t>0:X_t=a\bigg| X_0=x\right\}.
\]
According to \cite{revuz2013continuous}, $BES(1+\epsilon)$ is recurrent on $\mathcal{D}$ for all $\epsilon\in (-1,1)$. This guarantees $\tau$ is finite \emph{a.s.} Therefore the corresponding BVP exists. The BVP boundary in \eqref{bv0} is specified by $f(a)=1$. And refer to \cite{kent1978some}, the initial LT is given by the ratio of modified Bessel functions of the second kind. Similar as in the OU process, the initial LT can be solved numerically via \cite{abate2006unified} or using spectral decompositions \cite{linetsky2004computing}. We will repeat our previous work and solve the first order downward FPTD via perturbation.

\begin{lem}[First Order BVP Solution]\label{lemma51}
The first order perturbed LT is given by
\[
f_1(x,\beta)=f_0(x,\beta)\frac{\ln( \frac{a}{x})+e^{2\gamma a}E_1(2\gamma a)-e^{2\gamma x}E_1(2\gamma x)}{2},
\]
where $E_1(\cdot)$ is the \emph{exponential integral} function and 
\[
f_0(x,\beta)=e^{-\gamma (x-a)}.
\]
\end{lem}
\begin{proof}
Follow our recursion framework. Denote the LT of Brownian motion FPTD by $f_0(x,\beta)=e^{-\gamma (x-a)}$ \cite{borodin2012handbook}. Recall \eqref{eqnK}, $k_1$ is specified by $k_1=\gamma$. Follow \eqref{eqngeneralsolutiongi}, by assigning different lower integral-limits we have
\begin{equation}\label{eqn:interiorintegral}
g_1=\int_{0}^x 2e^{2\gamma y} \left[\gamma\int_{+\infty}^y \frac{1}{2z}e^{-2\gamma z}dz + C_1\right]dy +C_2.
\end{equation}
After standard calculations, the inner integral yields
\[
\frac{\gamma}{2}\int_{+\infty}^y \frac{1}{z}e^{-2\gamma z}dz=-\frac{\gamma}{2}E_1(2\gamma y).
\]
Simplifying \eqref{eqn:interiorintegral} we get
\begin{equation}\label{eqn:eqnabove11}
g_1=\int_{0}^x -\gamma e^{2\gamma y} E_1(2\gamma y)dy+ C_1\frac{e^{2\gamma x}-1}{\gamma} +C_2.
\end{equation}
Refer to \cite[Equation (5), 4.2.]{geller1969table}, for the integral in \eqref{eqn:eqnabove11} we have
\[
\int_{0}^x -\gamma e^{2\gamma y} E_1(2\gamma y)dy=-\frac{\ln( x)+e^{2\gamma x}E_1(2\gamma x)}{2}+C_3.
\]
Combine $C_3$ and $C_2$, \eqref{eqn:interiorintegral} then becomes
\[
g_1=-\frac{\ln( x)+e^{2\gamma x}E_1(2\gamma x)}{2}+C_1\frac{e^{2\gamma x}-1}{\gamma}+C_2.
\]

$C_1$ and $C_2$ are determined through boundary conditions in the $o(\epsilon^1)$-problem. We start with the upper boundary $\set{+\infty}$. By L'H\^opital's rule $f_0(x)\ln(x)\rightarrow 0$ as $x\uparrow +\infty$. And refer to \cite[Equation (5), 3.3]{geller1969table}, for the $E_1$-term we get
\[
\lim_{x\uparrow +\infty}f_0(x)e^{2\gamma x}E_1(2\gamma x)=0.
\]
Therefore, the only $C_1$ which satisfies the condition $f_1(+\infty)=0$ is $0$. Similarly, solving $f_1(a)=0$ gives 
\[
C_2=\frac{\ln(a)+e^{2\gamma a}E_1(2\gamma a)}{2}.
\]
We conclude the proof by here.
\end{proof}
\\

\begin{cor}[Full Integrability]\label{cor53}
The first order perturbed FPTD of $BES(1+\epsilon)$ is fully integrable.
\end{cor}
\begin{proof}
The corollary is proved by showing
\[
\lim_{\beta\downarrow 0^{+}}f^1(x,\beta)=1.
\]
Since $\lim_{\beta\downarrow 0^{+}}f_0(x,\beta)=1$, it is sufficient only to justify
\[
\lim_{\beta\downarrow 0^{+}}g_1(x,\beta)=0.
\]
Indeed, by considering \cite[Equation (1), 3.3]{geller1969table} that
\begin{equation*}\label{eqnE1expansion}
E_1(x)=-c-\ln(x)+\int_0^x \frac{1-e^{-u}}{u}du,
\end{equation*}
where $c=0.5772$ is the Euler's constant, one can demonstrate the result immediately.
\end{proof}
\\

\begin{prop}[First Order Perturbed FPTD of $BES(1+\epsilon)$]\label{prop:54}
The first order perturbed downward FPTD of $BES(1+\epsilon)$ is given by
\begin{equation}\label{secondtermreference}
p_\tau^{(1)}(t)=\left(1+ \frac{\epsilon}{2}\ln( \frac{a}{x})\right)p_{\tau,x,a}^{(0)}(t)+\frac{\epsilon(x-a)}{\sqrt{2\pi \cdot t}}\int_{(x-a)^2}^\infty  \frac{p_\Gamma(y,1,\frac{1}{2t})dy}{(\sqrt{y}-x+3a)(\sqrt{y}+x+a)}dy,
\end{equation}
where
\begin{equation}\label{eqn:new0def}
p_{\tau,x,a}^{(0)}(t):=\frac{x-a}{\sqrt{2\pi}}t^{-\frac{3}{2}}e^{-\frac{(x-a)^2}{2t}}
\end{equation}
and
\begin{equation}\label{eqnGfunctioncomment}
p_\Gamma(y,1,\frac{1}{2t}):=\frac{1}{2t}e^{-\frac{y}{2t}}
\end{equation}
are density functions of inverse Gamma and Gamma distributions, respectively.
\end{prop}

\begin{proof}
Rearrange terms in the first order LT,
\begin{equation}\label{eqn:above2}
f^1(x,\beta)=f_0(x,\beta)\left(1+ \frac{\epsilon}{2}\ln( \frac{a}{x})\right)+\epsilon f_0(x,\beta)\frac{e^{2\gamma a}E_1(2\gamma a)-e^{2\gamma x}E_1(2\gamma x)}{2}.
\end{equation}
The inverse of the first term in \eqref{eqn:above2} is given by inverse Gamma distribution \cite{borodin2012handbook}. Now consider the second term. Recall the definition of $E_1(z)$ that
\[
E_1(z):=\int_z^{+\infty}\frac{e^{-u}}{u}du.
\]
By multiplying $e^z$ we have 
\begin{equation}\label{eqn:integralLT}
e^zE_1(z)=\int_0^\infty\frac{e^{-uz}}{u+1}du. 
\end{equation}
Assume we can change the order of integrals, then substituting \eqref{eqn:integralLT} into the second term of \eqref{eqn:above2} yields
\begin{align}
\LT^{-1}\set{e^{-\sqrt{2\beta}(x-a)}e^{z\sqrt{2\beta}}E_1(z\sqrt{2\beta})}(t)&=\int_0^\infty \frac{\LT^{-1}\set{e^{-\sqrt{2\beta}(x-a+uz)}}(t)}{u+1}du \label{inverseofIntegralLTcon}\\
&=\int_0^\infty\frac{p^{(0)}_{\tau, x-a+uz, 0}(t)}{u+1}du \label{inverseofIntegralLT},
\end{align}
where $p^{(0)}_{\tau,  x-a+uz, 0}(t)$ is defined by \eqref{eqn:new0def}; more precisely, by letting 
\begin{equation}\label{y_substitution}
y:=(x-a+uz)^2,
\end{equation}
\eqref{inverseofIntegralLT} further gives 
\begin{equation}\label{finaleqnaxsub}
\LT^{-1}\set{e^{-\sqrt{2\beta}(x-a)}e^{z\sqrt{2\beta}}E_1(z\sqrt{2\beta})}(t)=\frac{1}{\sqrt{2\pi \cdot t}}\int_{(x-a)^2}^\infty \frac{p_\Gamma(y,1,\frac{1}{2t})}{\sqrt{y}-(x-a)+z}dy,
\end{equation}
where $p_\Gamma(y,1,\frac{1}{2t})$ is the density function of Gamma distribution with shape parameter $1$ and rate parameter $\frac{1}{2t}$. The rest of the proof is concluded by substituting $z=2a$ and $z=2x$ into \eqref{finaleqnaxsub}, respectively.
\end{proof}
\\

\begin{rmk}
The perturbed density for $BES(1+\epsilon)$ is continuous on $(0,+\infty)$. As a sufficient but not necessary condition, as long as\footnote{In fact, in the case $\epsilon<0$ (this gives a Bessel process of an order between $0$ and $1$) the inequality holds true for any $a\le x$.}
\[
1+\frac{\epsilon}{2}\ln(\frac{a}{x})\ge 0,
\]
the function $p_\tau^{(1)}(t)$ then is guaranteed to be positive. Therefore, combine corollary \ref{cor53}, the first order perturbed FPTD is a valid probability density function. On the other hand, by using properties of integer Bessel functions in \cite[Chapter III]{watson1995treatise}, potentially one can derive FPTDs of $BES(n+\epsilon)$ with $n\in\mathbb{Z}^{+}$.
\end{rmk}

\begin{prop}[Tail Asymptotics of $BES(1+\epsilon)$ Perturbed FPTD]
Tail asymptotics of the first order perturbed $BES(1+\epsilon)$ FPTD are given by
\begin{align*}
p_\tau^{(1)}(t)&\sim p_{\tau,x,a}^{(0)}(t),\ \text{as }t\downarrow 0^{+},\tag*{(Left Tail)}\\
p_\tau^{(1)}(t)&\sim O(t^{-1-\alpha}),\ \text{as }t\uparrow +\infty,\tag*{(Right Tail)}
\end{align*}
where $0<\alpha<\frac{1}{2}$.
\end{prop}

\begin{proof}
We start from the left tail. The first term in $p_\tau^{(1)}(t)$ has the same order as the FPTD of Brownian motion. We only need to check the second term, which involves an expectation of Gamma variable. Consider a convolution representation of \eqref{inverseofIntegralLT}. From \eqref{inverseofIntegralLTcon}, we have
\begin{equation}\label{eqngdefunction}
\frac{1}{\sqrt{2\pi \cdot t}}\int_{(x-a)^2}^\infty  \frac{p_\Gamma(y,1,\frac{1}{2t})dy}{(\sqrt{y}-x+3a)(\sqrt{y}+x+a)}dy=\int_0^t p^{(0)}_{\tau,x,a}(v) \int_0^\infty \frac{p_{\tau,2ua,0}^{(0)}(t-v)-p_{\tau,2ux,0}^{(0)}(t-v)}{2(u+1)}dudv.
\end{equation}
Denote by
\begin{equation}\label{eqn:defat}
\lambda(t):=t^{-\frac{3}{2}}e^{-\frac{(x-a)^2}{2t}}
\end{equation}
and
\begin{equation}\label{eqn:defbt}
\mu(t):=\int_0^t \lambda(v)\int_0^\infty \frac{p_{\tau,2ua,0}^{(0)}(t-v)-p_{\tau,2ux,0}^{(0)}(t-v)}{2(u+1)}dudv
\end{equation}
for asymptotics of the inverse Gamma density and the convolution, respectively. Let $t\downarrow 0^{+}$, then \eqref{eqn:defbt} yields
\begin{equation}\label{eqn:btasymptoticssmall}
\mu(t)\sim \int_0^t \lambda(v)\int_0^\infty \frac{p_{\tau,2ua,0}^{(0)}(v)-p_{\tau,2ux,0}^{(0)}(v)}{2(u+1)}dudv.
\end{equation}
Define 
\[
r(t):=\frac{\mu(t)}{\lambda(t)}.
\]
First, note that the limit of $r(t)$ at $0$ is $\frac{0}{0}$-type. To see this, consider the LT of $\mu(t)$. Based on the initial value theorem, we have
\[
\lim_{t\downarrow 0^{+}}\mu(t)=\lim_{\beta\uparrow +\infty}\left(\beta e^{-\sqrt{2\beta}(x-a)}\left(e^{2a\sqrt{2\beta}}E_1(2a\sqrt{2\beta})-e^{2x\sqrt{2\beta}}E_1(2x\sqrt{2\beta})\right)\right).
\]
According to \cite[6.12.1]{lozier2003nist}, we further get
\begin{equation}\label{eqn:above3}
\lim_{t\downarrow 0^{+}}\mu(t)=\lim_{\beta\uparrow +\infty}\left(\beta e^{-\sqrt{2\beta}(x-a)}\left(\frac{1}{2a\sqrt{2\beta}}-\frac{1}{2x\sqrt{2\beta}}\right)\right)=0.
\end{equation}
Now apply L'H\^opital's rule on $r(t)$. Based on the asymptotic in \eqref{eqn:btasymptoticssmall}, the limit of $r(t)$ is given by
\begin{align*}
\lim_{t\downarrow 0^{+}} \frac{\mu(t)}{\lambda(t)}=&\lim_{t\downarrow 0^{+}}\frac{\lambda(t)\int_0^\infty  \frac{p_{\tau,2ua,0}^{(0)}(t)-p_{\tau,2ux,0}^{(0)}(t)}{2(u+1)}du}{\lambda^{'}(t)}\\
=&\lim_{t\downarrow 0^{+}}\frac{t^{-\frac{6}{2}}\int_0^\infty \left(\frac{ua}{u+1}e^{-\frac{2u^2a^2}{t}}-\frac{ux}{u+1}e^{-\frac{2x^2a^2}{t}}\right)du}{t^{-\frac{7}{2}}}\\
=&0.
\end{align*}
Therefore $\mu(t)=o(\lambda(t))$, and the left tail asymptotic is $\left(1+ \frac{\epsilon}{2}\ln( \frac{a}{x})\right)p_{\tau,x,a}^{(0)}(t)$.

For the right tail, consider \eqref{eqngdefunction}, we further have
\[
\int_0^\infty \frac{p_{\tau,2ua,0}^{(0)}(t-v)-p_{\tau,2ux,0}^{(0)}(t-v)}{2(u+1)}du = t^{-\frac{3}{2}}\int_0^\infty \frac{w \phi(w)}{(w+\frac{2a}{\sqrt{}t})(w+\frac{2x}{\sqrt{t}})}dw\ge 0,
\]
where $\phi(w)$ is the density of standard normal distribution. Therefore,
\[
\lambda(v)\int_0^\infty \frac{p_{\tau,2ua,0}^{(0)}(t-v)-p_{\tau,2ux,0}^{(0)}(t-v)}{2(u+1)}du\ge 0,\ \forall v\in (0,t).
\]
Let $K<<t$ be a large but fixed constant, then
\begin{equation}\label{eqn:bknotation}
\mu(t)\ge\int_0^K \lambda(v)\int_0^\infty \frac{p_{\tau,2ua,0}^{(0)}(t-v)-p_{\tau,2ux,0}^{(0)}(t-v)}{2(u+1)}dudv.
\end{equation}
Denote the right-hand side of \eqref{eqn:bknotation} by $\mu_K(t)$. For $0\le v\le K<<t$, we further have
\begin{equation}\label{eqnasymptbt111}
\mu_K(t)\sim \int_0^\infty \frac{p_{\tau,2uz,0}^{(0)}(t)}{2(u+1)}du\sim t^{-\frac{3}{2}}\int_0^\infty \frac{u}{u+1}e^{-\frac{u^2}{t}}du.
\end{equation}
Since $t^{-\frac{3}{2}}$ is the asymptotic of $\lambda(t)$ for large $t$, and 
\[
\lim_{t\uparrow +\infty}\frac{\mu_K(t)}{t^{-\frac{3}{2}}}=\int_0^\infty \frac{u}{u+1}du=\infty,
\]
so according to \eqref{eqn:bknotation} we find
\begin{equation}\label{eqnatcomparison}
\lim_{t\uparrow +\infty}\frac{\mu(t)}{\lambda(t)}\ge \lim_{t\uparrow +\infty}\frac{\mu_K(t)}{\lambda(t)}=\infty.
\end{equation}
This yields
\[
p_\tau^{(1)}(t)\sim \mu(t).
\]

The next step is to confirm that $\mu(t)$ is a valid asymptotic, \emph{i.e.} itself does not diverge. Indeed, applying the final value theorem one immediately has
\[
\lim_{t\uparrow +\infty}\mu(t)=0.
\]
In addition, note that previously in corollary \ref{cor53}, we have shown 
\[
\lim_{\beta\downarrow 0^{+}}f_1(x,\beta)=\lim_{\beta\downarrow 0^{+}}f_0(x,\beta)g_1(x,\beta)=0,
\]
which further yields
\begin{equation}\label{invtempadded}
\int_{0}^\infty \LT^{-1}\set{f_1(x,\beta)}(t) dt = 0.
\end{equation}
Since the inverse in \eqref{invtempadded} is given by 
\[
\LT^{-1}\set{f_1(x,\beta)}(t)=\epsilon\left(\frac{1}{2}\ln(\frac{a}{x})p_{\tau,x,a}^{(0)}(t)+\frac{(x-a)^2}{\sqrt{2\pi}}\mu(t)\right),
\]
we further get
\begin{align}
\int_0^\infty \mu(t)dt =&\frac{\sqrt{2\pi}}{2(x-a)^2}\ln(\frac{x}{a})\int_0^\infty p_{\tau,x,a}^{(0)}(t)dt\label{eqnbounpre}\\
=&\frac{\sqrt{2\pi}}{2(x-a)^2}\ln(\frac{x}{a}).\label{eqnbound}
\end{align}
The equality \eqref{eqnbound} comes from the fact that $p_{\tau,x,a}^{(0)}(t)$ is the \emph{p.d.f.} of inverse Gamma distribution. As a necessary condition for \eqref{eqnbound} to hold true, the convergence rate of $\mu(t)$ should be
\[
\mu(t)=O(t^{-1-\alpha}),
\]
for some $\alpha>0$. Combining \eqref{eqnatcomparison}, in the end we show $0<\alpha < \frac{1}{2}$. 
\end{proof}
\\

\begin{prop}[Error Estimation and Convergence of $BES(1+\epsilon)$]\label{prop:prop46}
The $\eta$-function is given by
\begin{equation}\label{etaforBes}
\eta(x,t)=\frac{1}{2}\ln(\frac{x}{a})\rho(x,a,t)-\frac{1}{4t}\int_{(x-a)^2}^\infty \left(\frac{p_\Gamma(y,\frac{3}{2},\frac{1}{2t})-p_\Gamma(y,\frac{1}{2},\frac{1}{2t})}{\sqrt{y}-x+3a}+\frac{p_\Gamma(y,\frac{3}{2},\frac{1}{2t})-p_\Gamma(y,\frac{1}{2},\frac{1}{2t})}{\sqrt{y}+x+a}\right)dy,
\end{equation}
where
\[
\rho(x,a,t):=\frac{1}{\sqrt{2\pi}}t^{-\frac{5}{2}}\left((x-a)^2-t\right)e^{-\frac{(x-a)^2}{2t}};
\]
and $p_\Gamma(y,\alpha,s)$ is the density of Gamma distribution with shape parameter $\alpha$ and rate parameter $s$. Besides, for all $t\in [0,+\infty)$, the probabilistic representation \eqref{eqn:errorestimateqfunc} is valid. And for $t\in [0,T]$ with fixed $T>0$, the first order perturbed FPTD of $BES(1+\epsilon)$ converges at rate $o(\epsilon^2)$.
\end{prop}
\begin{proof}
Taking partial derivative on $f_1(x,\beta)$ yields
\[
\partial_xf_1(x,\beta)=-\gamma f_0(x,\beta)\frac{\ln(\frac{x}{a})+e^{2\gamma a}E_1(2\gamma a)+e^{2\gamma x}E_1(2\gamma x)}{2}.
\]
According to \cite{bateman1954tables}, for some positive $z$, we have
\[
\LT^{-1}\set{\sqrt{2\beta}e^{-\sqrt{2\beta}z}}(t)=\frac{1}{\sqrt{2\pi}}t^{-\frac{5}{2}}\left(z^2-t\right)e^{-\frac{z^2}{2t}}.
\]
Then by using the same trick in \eqref{inverseofIntegralLT}, we prove the result of $\eta$-function.

Consider the error estimation with the $h$-function shown in \eqref{hufnctionforBessel}. The first part (when $x\in [a,K]$ for some large $K$) proof is given similarly as in the proof of proposition \ref{firorderOU}. When $x>K$ and $K$ is large enough, we have $\ln(\frac{x}{a})\le (x-a)$. Therefore 
\begin{equation}\label{eqnaboundenessBES}
\left|\frac{1}{4x}\ln(\frac{x}{a})\rho(x,a,t)\right|\le \frac{1}{4K}\frac{1}{\sqrt{2\pi}}t^{-\frac{5}{2}}\left((x-a)^3+t(x-a)\right)e^{-\frac{(x-a)^2}{2t}}.
\end{equation}
And the boundedness of \eqref{eqnaboundenessBES} is given by the proof in lemma \ref{firorderOUlemma}. The only thing left then, is to show the Gamma-density part in \eqref{etaforBes} satisfies \eqref{eqn:eqnboundnesscondition} and \eqref{eqn:assumpboundnedness}. First note that $y\ge (x-a)^2>0$ and $x>a$, therefore
\[
\max\set{\frac{1}{\sqrt{y}-x+3a},\frac{1}{\sqrt{y}+x+a}}\le \frac{1}{2a}.
\]
For the rest of the integral, consider a shape parameter $\alpha>0$. Then for $x>K$ and such that $x-a>1$, given $y>(x-a)^2>1$, we have
\begin{align*}
\frac{1}{4t}\int_{(x-a)^2}^\infty p_\Gamma(y,\alpha,\frac{1}{2t})dy&=
\frac{1}{2\Gamma(\alpha)}\int_{(x-a)^2}^\infty \frac{1}{(2t)^{\alpha+1}}y^{\alpha-1}e^{-\frac{y}{2\alpha}}dy\\
&\le \frac{\Gamma(\alpha+1)}{2\Gamma(\alpha)}\int_{(x-a)^2}^\infty \frac{y^{\alpha}}{\Gamma(\alpha+1)(2t)^{\alpha+1}}e^{-\frac{y}{2\alpha}}dy\\
&\le \frac{\Gamma(\alpha+1)}{2\Gamma(\alpha)}.
\end{align*}
The last inequality comes from the \emph{c.d.f.} of Gamma distribution. Substituting $\alpha=\frac{3}{2}$ and $\alpha=\frac{1}{2}$ into the inequality above then concludes our proof.
\end{proof}
\\

\section{Numerical Examples on Downward OU FPTD}
In this section, we illustrate our framework via two sets of OU process exercises. Note that through our discussions the perturbed OU FPTD is the only one which may not converge, or converge slowly with a fat-tail asymptotic (see proposition \ref{prop:tailasmptotocsou}). Therefore it is more meaningful to check the effectiveness of the framework under this extreme scenario. More numerical results, including the application on the exponential-Shiryaev process \cite{dassios2018economic}, can be found in the appendix. 

We consider a generalised OU model with constant volatility:
\[
dX_t=\epsilon(\theta-X_t)dt + \sigma dW_t.
\]
The starting point is $X_0=x$ and hitting level is denoted by $l<x$. The affine transformation $Y_t=\frac{X_t-l}{\sigma}$ yields
\[
dY_t=\epsilon(\hat{\theta}-Y_t)dt+dW_t,
\]
where $Y_0=\frac{x-l}{\sigma}$ and $\hat{\theta}=\frac{\theta-l}{\sigma}$. By considering the new hitting level $\hat{l}=0$, the initial hitting time for $X_t$ from $x$ to $l$ is equivalent to the problem for $Y_t$ from $Y_0$ to $\hat{l}$. The $Y_t$-hitting time then can be solved explicitly via results in section \ref{sec4}. Apart from the perturbation method, there are many other studies in finding the FPTD of OU process. For benchmarking purpose we list them below
\begin{itemize}
\item[i).] Talbot method of numerical inverse LT \cite{abate2006unified}
\item[ii).] FPTD representation using spectral decomposition \cite{alili2005representations}
\item[iii).] 3-dimensional Brownian bridge simulation \cite{ichiba2011efficient}
\item[iv).] closed-form solution in the special case $l=\theta$ \cite{going2003clarification}
\end{itemize}

Note that the closed-form density in iv) does not involve any numerical approximation. Therefore we can treat it as the `true' density. In general cases where $l\neq \theta$, there is no closed-form solution found yet. As an alternative, we use Talbot approach to be our benchmark. Model parameters are selected as follows
\[
\epsilon = 0.1,\ \theta = 0.3,\ \sigma = 0.3,\ x = 0.5.
\]
We conduct two sets of exercises. In the first one we consider $l=0.3$. And in the second group, we let $l=0.2$.

\subsection{Benchmark Comparison for $l=\theta$}\label{sec51}
In this section, we only compare the Talbot method and first order perturbation with the closed-form solution. Figures \ref{fig1} and \ref{fig2} show densities and their relative errors (\emph{w.r.t.} approach iv)) in 5 years time.

Green dots in figure \ref{fig1} indicate `true' densities. The blue line and orange segment curves plot Talbot inverse and first order perturbation, respectively. From visual observations we find three densities coincide with each other. This confirms that our perturbed FPTD is valid.

 \begin{figure}[h]
 \centering
 \begin{minipage}[c]{0.48\textwidth}
 \centering
        \includegraphics[width=1.\textwidth, height = .8\textwidth]{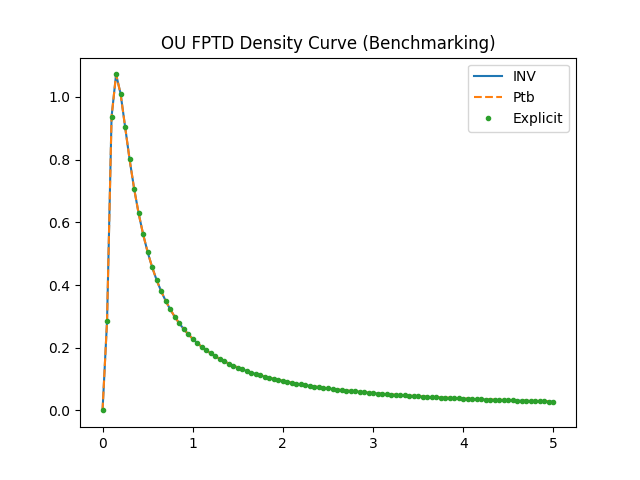}
        \caption{FPT density of OU process}
        \label{fig1}
 \end{minipage}
\hspace{1ex}
\begin{minipage}[c]{0.48\textwidth}
\centering
        \includegraphics[width=1.\textwidth, height = 0.8\textwidth]{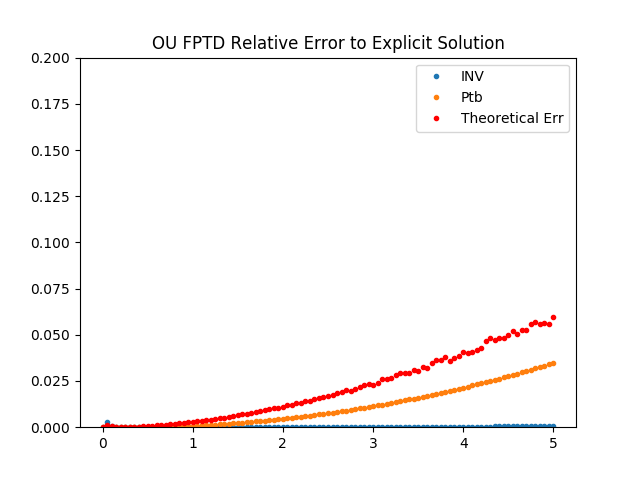}
        \caption{Relative error to iv)}
        \label{fig2}
 \end{minipage}
 \end{figure}

In order to compare accurateness and verify our error estimation formulae, we demonstrate relative errors in figure \ref{fig2}. Blue and yellow dots depict realised errors, which are calculated from Talbot inverse and perturbation density by benchmarking them on closed-form solution. Red dots are numerical evaluations from proposition \ref{prop21} and lemma \ref{lemma410}. We refer them to be theoretical errors. In the computation of $q_\tau(t)$ we simulate 1000 paths with $dt = \frac{t}{1000}$. And the theoretical error is calculated by
\[
err = \left|\frac{q_\tau(t)}{p^{(1)}_\tau(t)+q_\tau(t)}\right|.
\]

From figure \ref{fig2} we see in general the Talbot inverse is very accurate, apart from that there is a peak for small $t$. In fact, by checking figure \ref{fig6} in appendix, we find perturbation approach outperforms the Talbot inverse on the left tail. Considering the fact that perturbation provides smooth asymptotic on the left tail (proposition \ref{prop:tailasmptotocsou}), this result is not surprising. On the other hand, error function on perturbation diverges when $t$ becomes large. Although this is not very encouraging, it confirms our analysis for the right tail asymptotic. 

In figure \ref{fig2}, by further comparing theoretical and realised errors we find they increase at the same rate. This verifies that proposition \ref{prop21} provides reasonable estimate to the level of error. The spread of them could be explained by limitations from Euler's simulation scheme. By reducing $dt$ to $0$, theoretically, we should be able to match red and orange curves. However, in practice we are more interested in the error range rather than exact error values; otherwise, the problem becomes equivalent to solve FPTD via Monte Carlo simulations.

In terms of computation speed, we provide the time of evaluating 100 density points. Without considering the initialisation of $\set{c_k^{(i,j)}}$ sequence\footnote{The sequence does not involve any parameter in OU model. Therefore it is only initialised once and pre-saved in disk.}, the perturbation has the same speed as the explicit density. Both of them spend about 0.001 seconds. The Talbot inverse spends about 1.371 seconds - approximately 1371 times slower than other two methods. 

\subsection{General Case Comparison for $l\neq\theta$}
In the second exercise, we consider $l\neq \theta$ and illustrate results in the same way as before. Note that for $l\neq\theta$ we do not have any knowledge of the `true' density. Considering in last section the Talbot inverse is almost the same as the closed-form solution, we therefore use it for benchmarking.

 \begin{figure}[h]
 \centering
 \begin{minipage}[c]{0.48\textwidth}
 \centering
        \includegraphics[width=1.\textwidth, height = .8\textwidth]{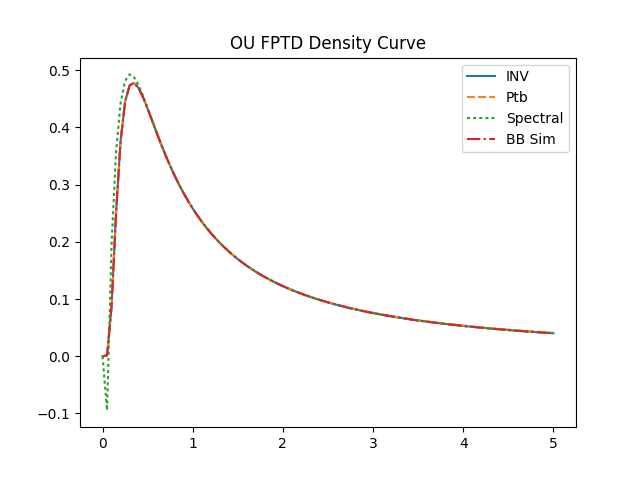}
        \caption{FPTD of OU process for general case}
        \label{fig3}
 \end{minipage}
\hspace{1ex}
\begin{minipage}[c]{0.48\textwidth}
\centering
        \includegraphics[width=1.\textwidth, height = 0.8\textwidth]{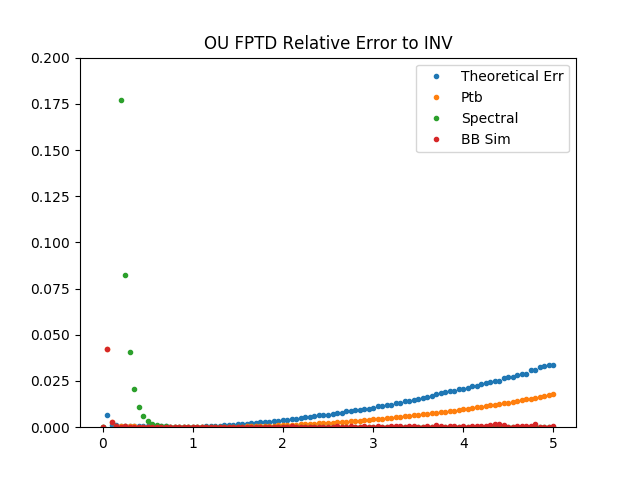}
        \caption{Relative error to i) for general case}
        \label{fig4}
 \end{minipage}
 \end{figure}

An immediate observation from figure \ref{fig3} is that the spectral decomposition does not provide a good estimate on the left tail. By checking \cite{alili2005representations} one can find this is due to the divergence of spectral series at $t=0$. Apart from the green curve, the rest three methods provide almost identical results. This confirms that the perturbation method works for a general case.

Results in figure \ref{fig4} are similar to what we have seen in section \ref{sec51}. We focus on explaining the spectral approach and Brownian bridge simulation. Observed from error plots, for large $t$ those two methods can provide very accurate estimates. However, we have to mention that their accuracy is based on the cost of requiring extra computational resources. In terms of the perturbation approach, although it is not as accurate as other methods (with a relative error $<2.5\%$), it maintains an overwhelming advantage in computation efficiency.

\section{Conclusion}
In this paper, we provide a systematic approach to solve closed-form asymptotics on FPTDs. We demonstrate the convergence in perturbation and derive probabilistic representation for error estimate. The perturbation resulted closed-form solution does not only increase computational efficiency; but also provides analytical tractability in understanding FPT distributions at extreme times. Using the framework we find valid approximations for Ornstein-Uhlenbeck, Bessel and exponential-Shiryaev FPTDs. In addition, by considering time-changes, results on Bessel process could be easily extended to the Cox-Ingersoll-Ross process. In the end, theoretical work has been verified by numerical exercises. Potential applications of this paper could be found in survival analysis, bond option pricing, and many others.

\clearpage
\appendix
\section{Recursive Structure for OU Perturbation}
\begin{cor}[Decomposition Structure I]
\label{cor2}
For $i=1$ and $i=2$, $\left\{c_k^{(i,j)}\right\}$ is explicitly given by 

\[
\begin{cases}
c_0^{(1,2)}=\frac{1}{2};\ c_0^{(1,1)}=\frac{1}{2};\ c_1^{(1,1)}=-1;\\
c_0^{(2,4)}=\frac{1}{8};\ c_0^{(2,3)}=\frac{1}{12};\ c_1^{(2,3)}=-\frac{1}{2};\ c_0^{(2,2)}=-\frac{1}{8};\ c_1^{(2,2)}=0;\ c_2^{(2,2)}=\frac{1}{2};\ c_0^{(2,1)}=-\frac{1}{8};\ c_1^{(2,1)}=\frac{1}{2};\ c_2^{(2,1)}=-\frac{1}{2}.
\end{cases}
\]
\end{cor}

\begin{cor}[Decomposition Structure II]
\label{cor3}
For $i\ge 3$, $\left\{c_k^{(i,j)} \right\}$ is recursively determined by 
\[
\begin{cases}
j=2i:&\ c_0^{(i,2i)}=\frac{c_0^{(i-1,2i-2)}}{2i}\\
\\
j=2i-1:&\ c_0^{(i,2i-1)}=\frac{1}{(2i-1)}c_0^{(i-1,2i-3)}-\frac{2i-3}{2(2i-1)}c_0^{(i-1,2i-2)};\ c_1^{(i,2i-1)}=\frac{1}{(2i-1)}\left(c_1^{(i-1,2i-3)}-c_0^{(i-1,2i-2)}\right)\\
\\
2<j<2i-1:&\ \begin{cases}
k=(2i-j)\wedge i:&\ \begin{cases}c_{2i-j}^{(i,j)}=\frac{1}{j}\left(c_{2i-j}^{(i-1,j-2)}-c_{2i-j-1}^{(i-1,j-1)}\right), &\text{ if $j>i$}\\
c_{i}^{(i,i)}=-\frac{1}{i}c_{i-1}^{(i-1,i-1)}, &\text{ if $j=i$}\\
c_{i}^{(i,j)}=\frac{1}{2}(j+1)c_i^{(i,j+1)}-\frac{1}{j}c_{i-1}^{(i-1,j-1)}+c_{i-1}^{(i-1,j)}, &\text{ if $j<i$}
\end{cases}\\
\\
0<k<(2i-j)\wedge i:&\ c_{k}^{(i,j)}=\frac{1}{2}(j+1)c_{k}^{(i,j+1)}+\frac{1}{j}c_{k}^{(i-1,j-2)}-\frac{j-1}{j}c_{k}^{(i-1,j-1)}\\
&\ \ \ \ \ \ \ \ \ \ -\frac{1}{j}c_{k-1}^{(i-1,j-1)}+c_{k-1}^{(i-1,j)}\\
\\
k=0:&\ c_{0}^{(i,j)}=\frac{1}{2}(j+1)c_0^{(i,j+1)}+\frac{1}{j}c_0^{(i-1,j-2)}-\frac{j-1}{j}c_0^{(i-1,j-1)}\\
\\
\end{cases}\\
\\
j=2:&\ \begin{cases}
k=i:&\ c_i^{(i,2)}=\frac{3}{2}c_i^{(i,3)}-\frac{1}{2}c_{i-1}^{(i-1,1)}+c_{i-1}^{(i-1,2)}\\
\\
0<k<i:&\ c_k^{(i,2)}=\frac{3}{2}c_k^{(i,3)}-\frac{1}{2}c_k^{(i-1,1)}-\frac{1}{2}c_{k-1}^{(i-1,1)}+c_{k-1}^{(i-1,2)}\\
\\
k=0:&\ c_0^{(i,2)}=\frac{3}{2}c_0^{(i,3)}-\frac{1}{2}c_0^{(i-1,1)}
\end{cases}\\
\\
j=1:&\ \begin{cases}
k=i:&\ c_i^{(i,1)}=c_i^{(i,2)}+c_{i-1}^{(i-1,1)}\\
\\
0<k<i:&\ c_k^{(i,1)}=c_k^{(i,2)}+c_{k-1}^{(i-1,1)}\\
\\
k=0:&\ c_0^{(i,1)}=c_0^{(i,2)}
\end{cases}\\
\end{cases}.
\]
\end{cor}

\section{Supplementary Proof of Proposition \ref{firorderOU}}
\begin{lem}\label{firorderOUlemma}
Let $h(x)\eta(x,t)$ be given in the proof of proposition \ref{firorderOU} and $\tau$ be the downward FPT of OU process crossing $0$. For a constant $K>1$, we have
\[
\mathbb{E}_x\left[\int_0^{\tau\wedge t}\bold{1}_{\set{X_u> K}}|h(X_u)\eta(X_u,t-u)|du\right]\le C_0+C_1\sqrt{t}+C_2 t,
\]
where $C_0,\ C_1$ and $C_2$ are positive constants.
\end{lem}
\begin{proof}
According to \eqref{enq:etaforOU}, \eqref{eqn:D0}, \eqref{eqn:D1} and \cite[12.7.2]{lozier2003nist}, we can write
\begin{equation}\label{mxfunctiondef}
h(x)\eta(x,t-u)=\frac{\theta - x}{\sqrt{2\pi}}e^{-\frac{x^2}{2(t-u)}}\left(\frac{x\partial_x l_0}{(t-u)^{\frac{3}{2}}}-\frac{l_0}{(t-u)^{\frac{3}{2}}}\left(\frac{x^2}{t-u}-1\right)+\frac{\partial_x l_1}{(t-u)^{\frac{1}{2}}}-\frac{xl_1}{(t-u)^{\frac{3}{2}}}\right).
\end{equation}
For $r\in \mathbb{Z}^+$ and $x\ge 0$, we find
\begin{equation}\label{eqn:boundrex}
\left|e^{-\frac{x^2}{2(t-u)}}x^r\right|\le r^{\frac{r}{2}}\cdot (t-u)^{\frac{r}{2}},
\end{equation}
and $x^*=\sqrt{r(t-u)}$ is the turning point of $e^{-\frac{x^2}{2(t-u)}}x^r$. Based on \eqref{mxfunctiondef}, \eqref{eqn:boundrex} and \eqref{defLn}, for $(x,t)\in \mathcal{D}\times [0,+\infty)$ and $u\in [0,t]$, we get 
\begin{equation}\label{inequalitysubstut}
\left|h(x)\eta(x,t-u)\bold{1}_{\set{x>K}}\right|\le C^{'}_0+C^{'}_1\frac{1}{\sqrt{t-u}}+\left(C^{'}_2\frac{x}{(t-u)^{\frac{3}{2}}}+C^{'}_3\frac{x^3}{(t-u)^{\frac{5}{2}}}\right)e^{-\frac{x^2}{2(t-u)}}\bold{1}_{\set{x>K}},
\end{equation}
where $\set{C^{'}_i:i=0,1,...,3}$ are positive constants. Note that, substituting \eqref{eqn:boundrex} into last two terms in \eqref{inequalitysubstut} gives $\frac{1}{t-u}$, which diverges in the following integral. Therefore, extra care is required.

For $\alpha>0$, we consider a more general function:
\begin{equation}\label{implicitmdef}
m^{r,\alpha}(x,u,t):=\frac{x^r}{(t-u)^{1+\alpha}}e^{-\frac{x^2}{2(t-u)}}.
\end{equation}
First note that, when $t\le\frac{K^2}{r}$, for all $u\in[0,t]$, we have $\sqrt{r(t-u)}\le K<x$. By considering the monotonicity of $m^{r,\alpha}(x,u,t)$ after the turning point, we find
\begin{align}
m^{r,\alpha}(x,u,t)\bold{1}_{\set{x>K}}&\le
\frac{K^r}{(t-u)^{1+\alpha}}e^{-\frac{K^2}{2(t-u)}}\bold{1}_{\set{x>K}}\label{split1fortandK1}\\
&=2^\alpha\Gamma(\alpha)K^{r-2\alpha}IG(t-u,\alpha,\frac{K^2}{2})\bold{1}_{\set{x>K}}\label{split1fortandK2},
\end{align}
where $\Gamma(\cdot)$ is the Gamma function, and $IG(t,\alpha,\frac{K^2}{2})$ is the density function of inverse Gamma distribution with shape parameter $\alpha$ and scale parameter $\frac{K^2}{2}$. On the other hand, when $t>\frac{K^2}{r}$, we consider maximums from two separate intervals. If $u\in [t-\frac{K^2}{r}, t]$, by noticing $\sqrt{r(t-u)}<K<x$, we find \eqref{split1fortandK2} is still valid. If, otherwise, $u\in [0,t-\frac{K^2}{r})$, again, using \eqref{eqn:boundrex} yields
\begin{equation}
m^{r,\alpha}(x,u,t)\bold{1}_{\set{x>K}}\le \frac{r^{\frac{r}{2}}}{(t-u)^{1+\alpha-\frac{r}{2}}}\bold{1}_{\set{x>K}}\label{split1fortandK3}.
\end{equation}
Further, if $1+\alpha-\frac{r}{2}>0$, then for $u\in [0,t-\frac{K^2}{r})$,
\begin{equation}
m^{r,\alpha}(x,u,t)\bold{1}_{\set{x>K}}\le \frac{r^{1+\alpha}}{K^{2+2\alpha-r}}\bold{1}_{\set{x>K}}\label{split1fortandK4}.
\end{equation}
Follow \eqref{split1fortandK2} and \eqref{split1fortandK4}, by taking integral of \eqref{implicitmdef} and considering the conditional expectation, we get
\begin{align*}
\mathbb{E}_x\left[\int_0^tm^{r,\alpha}(X_u,u,t)\bold{1}_{\set{X_u>K}}du\right]&=\left(\bold{1}_{\set{t\le \frac{K^2}{r}}}+\bold{1}_{\set{t> \frac{K^2}{r}}}\right)\cdot \int_0^t\mathbb{E}_x\left[m^{r,\alpha}(X_u,u,t)\bold{1}_{\set{X_u>K}}\right]du\\
&\le 2^\alpha\Gamma(\alpha)K^{r-2\alpha}\int_0^tIG(t-u,\alpha,\frac{K^2}{2})\mathbb{P}_x(X_u>K)du\cdot \bold{1}_{\set{t\le \frac{K^2}{r}}}\\
&\ \ \ \ +2^\alpha\Gamma(\alpha)K^{r-2\alpha}\int_{t-\frac{K^2}{r}}^tIG(t-u,\alpha,\frac{K^2}{2})\mathbb{P}_x(X_u>K)du\cdot \bold{1}_{\set{t> \frac{K^2}{r}}}\\
&\ \ \ \ +\int_0^{t-\frac{K^2}{r}}\frac{r^{1+\alpha}}{K^{2+2\alpha-r}} \mathbb{P}_x(X_u>K)du\cdot \bold{1}_{\set{t> \frac{K^2}{r}}}\\
&\le \left(2^\alpha\Gamma(\alpha)K^{r-2\alpha}-\frac{r^\alpha}{K^{2\alpha -r}}\right)+\frac{r^{1+\alpha}}{K^{2+2\alpha-r}}\cdot t.
\end{align*}
The last inequality is true as $IG(t,\cdot,\cdot)$ is a density function and $\mathbb{P}_x(X_u>K)\le 1$. Now, combine the inequality above and \eqref{inequalitysubstut}, for the original problem we get
\begin{align*}
\mathbb{E}_x\left[\int_0^{\tau\wedge t}\bold{1}_{\set{X_u> K}}|h(X_u)\eta(X_u,t-u)|du\right]&\le \mathbb{E}_x\left[\int_0^{t}\bold{1}_{\set{X_u> K}}|h(X_u)\eta(X_u,t-u)|du\right]\\
&\le C_0+C_1\sqrt{t}+C_2 t.
\end{align*}
This concludes our proof.
\end{proof}
\clearpage

\section{Further Numerical Results}
\subsection{Left Tail Zoom-in for OU Process}
 \begin{figure}[h]
 \centering
 \begin{minipage}[c]{0.48\textwidth}
 \centering
        \includegraphics[width=1.\textwidth, height = .9\textwidth]{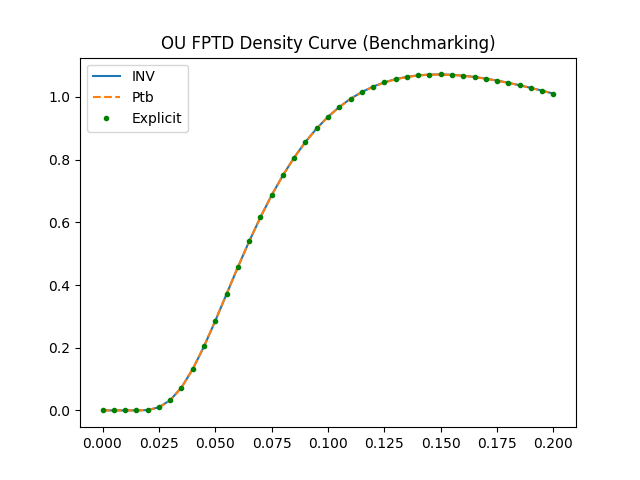}
        \caption{OU left tail density for $l= \theta$}
        \label{fig5}
 \end{minipage}
\hspace{1ex}
\begin{minipage}[c]{0.49\textwidth}
\centering
        \includegraphics[width=1.\textwidth, height = 0.9\textwidth]{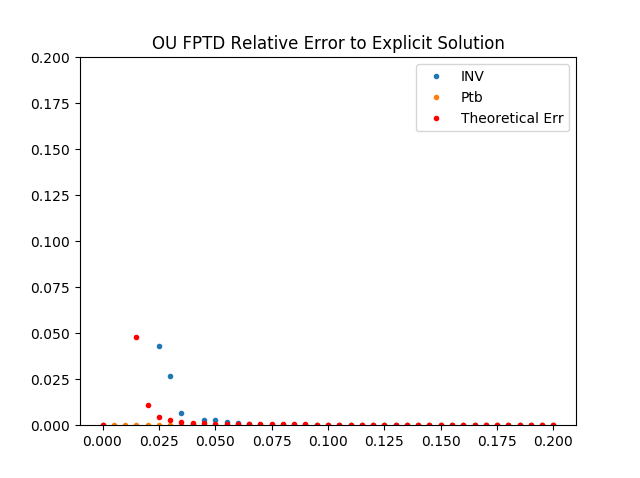}
        \caption{OU left tail error for $l= \theta$}
        \label{fig6}
 \end{minipage}
 \end{figure}

 \begin{figure}[h]
 \centering
 \begin{minipage}[c]{0.48\textwidth}
 \centering
        \includegraphics[width=1.\textwidth, height = .9\textwidth]{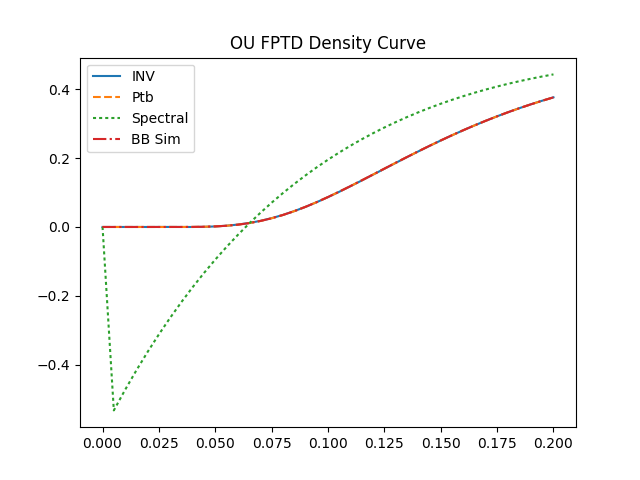}
        \caption{OU left tail density for $l\neq \theta$}
        \label{fig7}
 \end{minipage}
\hspace{1ex}
\begin{minipage}[c]{0.49\textwidth}
\centering
        \includegraphics[width=1.\textwidth, height = 0.9\textwidth]{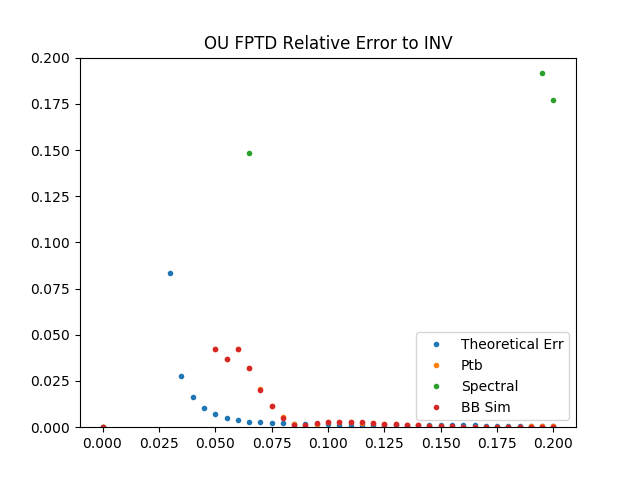}
        \caption{OU left tail error for $l\neq \theta$}
        \label{fig8}
 \end{minipage}
 \end{figure}
\clearpage

\subsection{Numerical Results for Bessel Process}
Error results in this section and section \ref{sectionnext} show theoretical errors are smaller than realised errors. This is due to the fact that Talbot inverse itself has numerical errors. From the results we may conclude the perturbation is more accurate than the Talbot inverse for $BES(1.5)$ and exponential-Shiryaev processes.

 \begin{figure}[h]
 \centering
 \begin{minipage}[c]{0.48\textwidth}
 \centering
        \includegraphics[width=1.\textwidth, height = .8\textwidth]{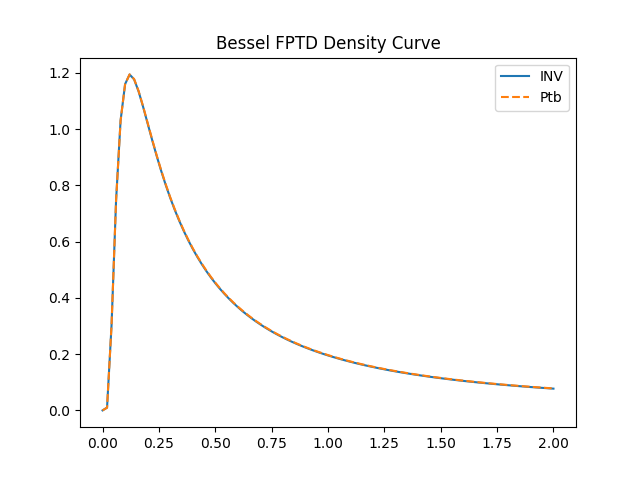}
        \caption{Bessel density with $\epsilon = 0.1,\ x=0.7,\ l=0.1$}
        \label{fig9}
 \end{minipage}
\hspace{1ex}
\begin{minipage}[c]{0.49\textwidth}
\centering
        \includegraphics[width=1.\textwidth, height = 0.8\textwidth]{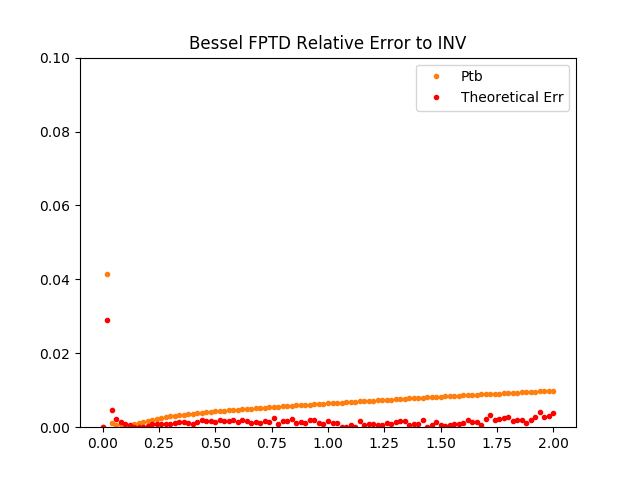}
        \caption{Bessel density error with $\epsilon = 0.1,\ x=0.7,\ l=0.1$}
        \label{fig10}
 \end{minipage}
 \end{figure}

 \begin{figure}[h]
 \centering
 \begin{minipage}[c]{0.48\textwidth}
 \centering
        \includegraphics[width=1.\textwidth, height = .9\textwidth]{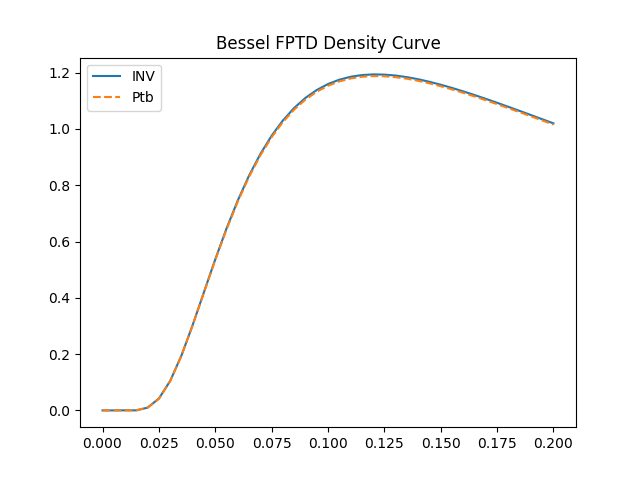}
        \caption{Bessel left tail density}
        \label{fig11}
 \end{minipage}
\hspace{1ex}
\begin{minipage}[c]{0.49\textwidth}
\centering
        \includegraphics[width=1.\textwidth, height = 0.9\textwidth]{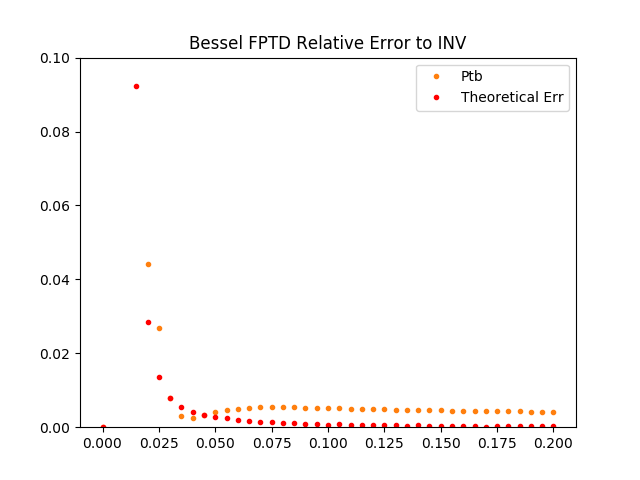}
        \caption{Bessel left tail relative error}
        \label{fig12}
 \end{minipage}
 \end{figure}
\clearpage

\subsection{Numerical Results for Exponential-Shiryaev Process}\label{sectionnext}

 \begin{figure}[h]
 \centering
 \begin{minipage}[c]{0.48\textwidth}
 \centering
        \includegraphics[width=1.\textwidth, height = .8\textwidth]{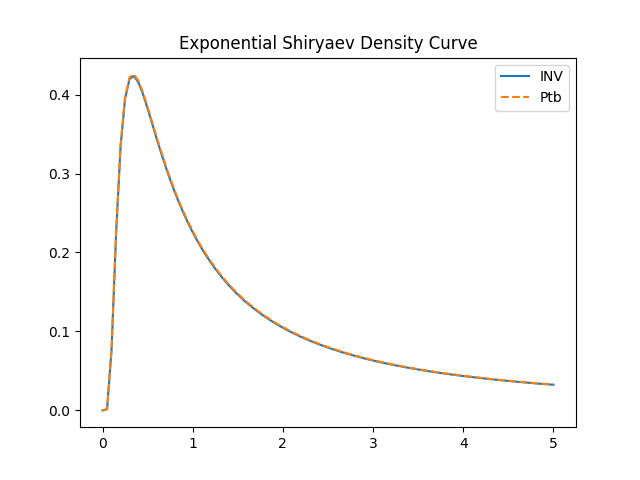}
        \caption{Exponential-Shiryaev density with $\epsilon = 0.1,\ \alpha = 0.5,\ c = 0.2,\ \sigma = 0.5,\ x=0.7,\ l=0.2$}
        \label{fig13}
 \end{minipage}
\hspace{1ex}
\begin{minipage}[c]{0.48\textwidth}
\centering
        \includegraphics[width=1.\textwidth, height = 0.8\textwidth]{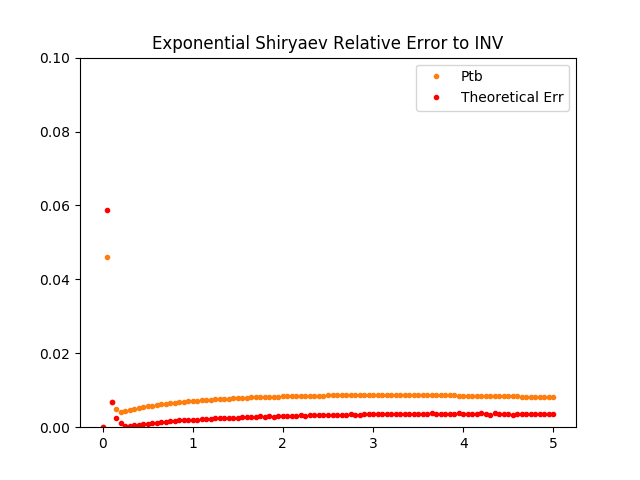}
        \caption{Density error with $\epsilon = 0.1,\ \alpha = 0.5,\ c = 0.2,\ \sigma = 0.5,\ x=0.7,\ l=0.2$}
        \label{fig14}
 \end{minipage}
 \end{figure}

 \begin{figure}[h]
 \centering
 \begin{minipage}[c]{0.49\textwidth}
 \centering
        \includegraphics[width=1.\textwidth, height = .9\textwidth]{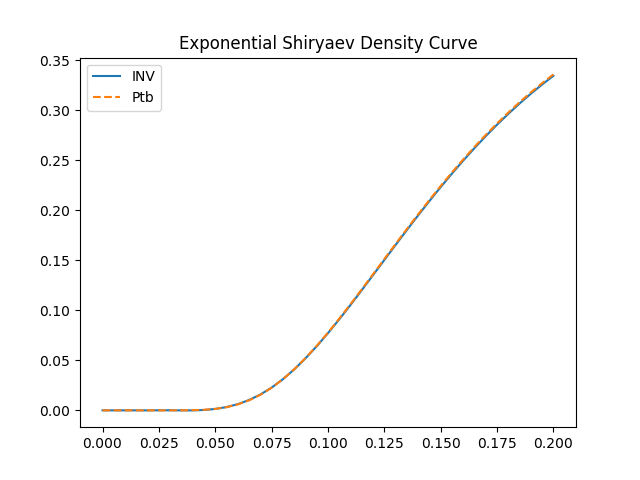}
        \caption{Exponential-Shiryaev left tail density}
        \label{fig15}
 \end{minipage}
\hspace{1ex}
\begin{minipage}[c]{0.48\textwidth}
\centering
        \includegraphics[width=1.\textwidth, height = 0.9\textwidth]{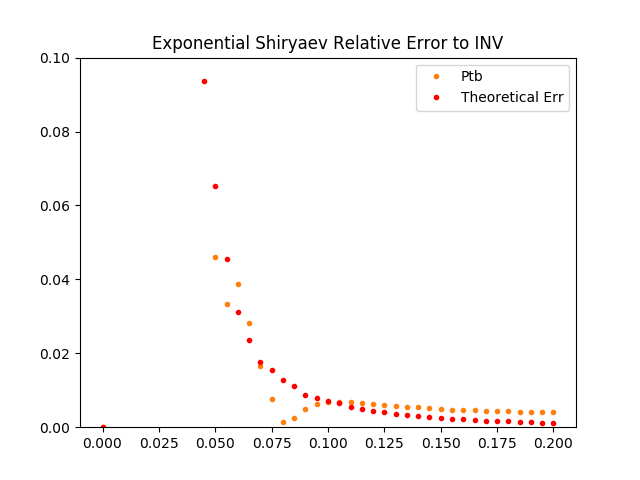}
        \caption{Left tail density error}
        \label{fig16}
 \end{minipage}
 \end{figure}

\clearpage

\bibliographystyle{plain}
\bibliography{reff}

\end{document}